\documentclass[11pt]{amsart}

\usepackage{epsfig}

\usepackage[utf8]{inputenc}
\usepackage[T1]{fontenc}

\usepackage{todonotes}
\usepackage{amssymb}
\usepackage{enumerate}
\usepackage{tikz-cd}
\usepackage{comment}
\usepackage{dsfont}

\usepackage{titlesec}
\titleformat{\section}{}{\thesection.}{0.8em}{\centering\scshape}
\titleformat{\subsection}{}{\thesubsection}{0.8em}{\scshape}

\usepackage[utf8]{inputenc}
\usepackage[T1]{fontenc}
\usepackage[english]{babel}
\usepackage[autostyle]{csquotes}

\usepackage{mathtools}
\usepackage{graphicx}
\graphicspath{ {./images/} }
\usepackage{caption}
\usepackage{subcaption}
\usepackage{xcolor}

\usepackage[shortlabels]{enumitem}
\usepackage[all]{xy}
\usepackage[colorlinks=true, urlcolor=rltblue, citecolor=drkgreen, linkcolor=drkred] {hyperref}
\usepackage{color}
\usepackage{pdfsync}
\definecolor{rltblue}{rgb}{0,0,0.4}
\definecolor{drkred}{rgb}{0.6,0,0}
\definecolor{drkgreen}{rgb}{0,0.4,0}
\usepackage{graphicx}
\usepackage{mathdots}
\usepackage{enumitem}
\setenumerate[1]{label=\arabic*.}

\newcommand{\orcidlink}[1]{%
  \href{https://orcid.org/#1}{https://orcid.org/#1}%
}

\usepackage{tikz-cd}
\usepackage{stmaryrd}
\usepackage{comment}
\usepackage{multicol}
\usepackage{adjustbox}

\newtheorem{mtheorem}{Theorem}
\newtheorem{theorem}{Theorem}[section]
\newtheorem{lemma}[theorem]{Lemma}

\newtheorem{corollary}[theorem]{Corollary}
\newtheorem{proposition}[theorem]{Proposition}

\theoremstyle{definition}
\newtheorem{example}[theorem]{Example}
\newtheorem{definition}[theorem]{Definition}

\newtheorem{remark}[theorem]{Remark}


\usepackage{lipsum}

\newcommand{\numberset}{\mathbb}
\newcommand{\N}{\numberset{N}}

\newcommand{\B}{\mathsf{B}}
\newcommand{\I}{\mathsf{I}}

\newcommand{\DBp}{\mathsf{DB_p}}
\newcommand{\DBpp}{\mathsf{DB^+_p}}
\newcommand{\DBi}[1]{\mathsf{DBi_{#1}}}
\newcommand{\DBin}{\DBi n}
\newcommand{\DBio}{\DBi 1}
\newcommand{\DBit}{\DBi 2}
\newcommand{\DBith}{\DBi 3}
\newcommand{\DBc}[1]{\mathsf{DBc_{#1}}}
\newcommand{\DBcn}{\DBc n}

\newcommand{\RCA}{\mathsf{RCA}_0}
\newcommand{\WKL}{\mathsf{WKL}_0}

\newcommand{\PI}[2]{\ensuremath{\boldsymbol\Pi^{#1}_{#2}}}
\newcommand{\SI}[2]{\ensuremath{\boldsymbol\Sigma^{#1}_{#2}}}
\newcommand{\DE}[2]{\ensuremath{\boldsymbol\Delta^{#1}_{#2}}}

\newcommand{\ran}{\operatorname{ran}}

\title{Reverse mathematics and dimension of posets}

\author{Alberto Marcone}
\address{
Dipartimento di Scienze Matematiche, Informatiche e Fisiche,\\
Universit\`a di Udine, Italy\\
ORCID: \orcidlink{0000-0001-8356-0086}
}
\email{alberto.marcone@uniud.it}
\urladdr{http://users.dimi.uniud.it/\textasciitilde alberto.marcone/}

\author{Andrea Volpi}
\address{
Dipartimento di Scienze Matematiche, Informatiche e Fisiche,\\
Universit\`a di Udine, Italy\\
ORCID: \orcidlink{0009-0008-2186-4103}
}
\curraddr{Department of Philosophy, University of Warsaw, Poland}
\email{andrea.volpi@uniud.it}
\urladdr{https://andreasdfghj.github.io/andreavolpi/}


\thanks{Both authors were supported by the Italian PRIN 2022 ``Models, sets and classifications'', prot.\ 2022TECZJA.
Marcone was also supported by the Marie Sklodowska-Curie Staff Exchanges Grant Agreement n.101236610: New Frontiers for Computability.
We are also thankful to the Erwin Schrödinger International Institute for Mathematics and Physics (ESI) for hosting us during the Reverse Mathematics thematic program of summer 2025.
Marcone is a member of INdAM-GNSAGA}

\thanks{The early stages of this research were carried out by the first author and Marta Fiori-Carones, whom both authors thank.}

\thanks{We thank also Arno Pauly, who allowed us to include his proofs in the appendix, and Leszek Kołodziejczyk for useful discussions about non-standard models.
We are grateful to the anonymous referee for a careful reading of the paper and especially for the comments that lead us to clarify the status of the existence of dimension in weak subsystems of second order arithmetic.}

\begin{document}

\begin{abstract}
Order dimension theory measures the complexity of partially ordered sets by quantifying how far they are from being linearly ordered. 
In this paper we study classical bounding results for order dimension within the framework of reverse mathematics. 
We focus on principles asserting that the dimension of a poset can be bounded in terms of the dimension of subposets obtained by removing chains or points, denoted by $\DBin$, $\DBcn$, and $\DBp$. 
We prove that, over $\RCA$, both $\DBin$ and $\DBcn$ are equivalent to $\WKL$. 
To analyze $\DBp$, we introduce a natural strengthening $\DBpp$ and show that both $\DBp$ and $\DBpp$ are provable from $\WKL$ and from $\I\SI02$, while $\B \SI02$ does not suffice to prove $\DBpp$.
The latter result is obtained by showing that the statement \lq\lq $\DBpp$ is computably true\rq\rq\ is equivalent to $\I\SI02$.
\end{abstract}

\maketitle

\section{Introduction}

\textit{Order dimension theory} lies at the intersection of \textit{order theory} and \textit{combinatorics}, and it provides a framework for understanding the complexity of \textit{partially ordered sets} (posets). 
Informally, the dimension measures how \lq\lq far\rq\rq\ a poset is from being linearly ordered. 
The greater the number of linear extensions required to describe a poset, the more intricate its structure is.

The study of poset dimension began with the seminal work of Dushnik and Miller in \cite{dushnikmiller}, where the concept was introduced. 
After that, research in the area was pursued by Trotter, Hiraguchi, and many other researchers who explored combinatorial and structural properties of poset dimension. 
The field expanded rapidly through the 1970s and 1980s (see e.g. \cite{bakerfishburnroberts, bogarttrotter, bogart, trotter, rabinovitch2, rabinovitch1}). 
The monograph by Trotter \cite{trotter2} became a central reference in the area, collecting and organizing many results.
More recent papers in the area include \cite{FelsnerMutzeWittmann, bergman2025}.



Several structural parameters are related to dimension, such as cardinality, height and width.
For example, the dimension of a poset $(P, \preceq)$, is bounded above both by its width (see \cite{dilworth}) and by $|P|/2$ (see \cite{hiraguchi}).
These bounds are tight: the standard example of a poset of dimension $n$ consists of the 1-element and $(n-1)$-element subsets of an $n$-element set ordered by inclusion: it has $2n$ elements and dimension $n$.
In Definition \ref{fn} we call this poset $F_n$ and then we analyze its properties.

Other significant contributions come from work of Baker, who showed that the dimension of a poset is at least as great as its breadth (a proof can be found in \cite{fishburn}). 
This provides a lower bound, complementing the upper bounds by Dilworth and Hiraguchi. 

An important feature of dimension is its monotonicity: a subposet cannot have greater dimension than the parent poset. 
In addition, small changes to a poset cannot drastically alter its dimension. 
These features make the dimension robust under typical constructions and modifications of posets.

We are interested in the connection between order theory and the \textit{foundations of mathematics} through the framework of \textit{reverse mathematics}.
In this context, researchers analyze the logical strength of mathematical theorems by determining the minimal axiomatic systems needed to prove them. 
Many results in order theory — such as variants of Dilworth's theorem, initial interval separation for posets, $\mathsf{WQO}$ and $\mathsf{BQO}$ theory — have been studied within subsystems of second-order arithmetic (see e.g.\ \cite[\S9.2]{dzhafarovmummert}).
More specifically results about linearizations of posets are included in \cite{cholakmarconesolomon, FM12}.

In this paper we aim to study classical bounding results about order dimension from the point of view of reverse mathematics.
In particular, we focus on theorems that link the dimension of a poset to the dimension of a subposet obtained by removing chains (the statements we call $\DBin$ and $\DBcn$) or points (called $\DBp$).

Our main results are the following:

\begin{mtheorem}\label{main1}
Within $\RCA$, $\WKL$ is equivalent to both $\DBin$ and $\DBcn$ for every $n \geq 1$, as well as to both $\forall n\, \DBin$ and $\forall n\, \DBcn$.
\end{mtheorem}

To better analyze $\DBp$ we introduce its strengthening $\DBpp$, which is actually obtained in the usual proof of $\DBp$. 

\begin{mtheorem}\label{main2}
$\DBp$ and $\DBpp$ are proved by $\WKL$ and also by $\I\SI02$, and thus not equivalent to any of them.
$\B \SI02$ does not suffice to prove $\DBpp$.
\end{mtheorem}

The last part of Theorem \ref{main2} is obtained by exploiting the \lq\lq computably true\rq\rq\ version of the statement, a tool first used in \cite{simpsonyokoyama} and recently exploited by Le Houérou and Patey \cite{patey2025}.\smallskip


Section \ref{sec:preliminaries} is used to fix the notation and give the basics of order theory and reverse mathematics required.
In Section \ref{sec:boundth} we state the bounding theorems $\DBin$, $\DBcn$ and $\DBp$ and we provide examples to show that these bounds are sharp.
Section \ref{sec:chains} is devoted to the proof of Theorem \ref{main1} (Theorem \ref{reversaldbcn}).
In Section \ref{sec:point} we study $\DBp$ and prove Theorem \ref{main2} (Corollaries \ref{dbpwkl} and \ref{reversemathresult}).

In the appendix we briefly discuss one of the statements introduced in Section \ref{sec:point} within the framework of Weihrauch reducibility.
The main result was proved by Arno Pauly after hearing a talk about this project.

\section{Preliminaries}\label{sec:preliminaries}

%
%
%
%


%
%

For a deeper introduction to the topics of order theory we refer to \cite{fishburn, harzheim}.

A \textit{partially ordered set} (or simply a \textit{poset}) is a set $P$, called the domain, together with a binary relation $\preceq$ which is reflexive, transitive and antisymmetric.
We denote by $\prec$  the irreflexive version of $\preceq$. 
We write $x \mid y$ when $x \npreceq y$ and $y \npreceq x$. 
We say that subsets $Y$ and $Z$ of a poset $(P,\preceq)$ are \textit{incomparable} if for each $y \in Y$ and each $z \in Z$, $y \mid z$.

A subset $X$ of a poset $(P, \preceq)$ is \textit{downward closed} if for every $x \in X$ and every $y \in P$, if $y \preceq x$ then $y \in X$.
We also say that $X$ is an \textit{initial interval} of $(P, \preceq)$.

A \textit{linear order} (or a \textit{chain}) is a poset $(P, \trianglelefteq)$ in which any two distinct elements are comparable.
We use the symbol $\preceq$ to denote a generic poset and we reserve the symbol $\trianglelefteq$ if we want to highlight that it is a linear order.
We say that a poset $(P,\preceq_1)$ extends a poset $(P,\preceq_2)$ if ${\preceq_2} \subseteq {\preceq_1}$.
An extension of a poset to a linear order is called a \textit{linearization}.

\begin{theorem}\label{szpilrajn}
Every poset can be linearized.
\end{theorem}

The classical proof can be found in \cite{szpilrajn} and uses Zorn's lemma.
Theorem \ref{szpilrajn} implies the axiom of finite choice (see \cite{horst}) which states that if $(S_\alpha)_{\alpha \in I}$ is a family of nonempty finite sets then the set theoretic product $\Pi_{\alpha \in A} S_\alpha$ is nonempty.
This statement is strictly weaker than full axiom of choice but it is still independent from Zermelo Fraenkel set theory (see \cite{moore}).
In \cite{howardrubin} it is shown that Theorem \ref{szpilrajn} combined with the statement that every total order has a cofinal well order, proves the full axiom of choice.

A number of parameters, such as \textit{height} (the maximum cardinality of a chain) and \textit{width} (the maximum cardinality of an antichain) describe a partial order.
We are mostly interested in dimension.

\begin{definition}\label{dimension}
If $(P, \preceq)$ is a poset, we say that a set $(P, \trianglelefteq_i)_{i \in I}$ of linearizations \textit{realize} $(P,\preceq)$ if $\bigcap_{i \in I} {\trianglelefteq_i} = {\preceq}$.
The \textit{dimension of $(P,\preceq)$} is the least cardinality of a realization and is denoted $\dim(P,\preceq)$ or simply by $\dim(P)$.
\end{definition}

By Theorem \ref{szpilrajn} there exists at least a linearization of any poset.
Moreover, if we can find a realization of a poset then the minimum described in Definition \ref{dimension} exists.
A poset has dimension $1$ if and only if it is a chain.
On the other hand, an antichain has dimension $2$: it suffices to take any linearization $\trianglelefteq$ of the antichain and its reverse $\trianglerighteq$.
A fundamental result proved by Dushnik and Miller (\cite{dushnikmiller}) states that for any cardinal $\kappa >0$ there exists a poset of dimension $\kappa$.

In practice, to show that $\dim(P) \leq \kappa$ it suffices to find a set of $\kappa$ linearizations $\{\trianglelefteq_\alpha : \alpha < \kappa\}$ of $P$ satisfying: for all $x,y \in P$ such that $x \mid y$ there exists $\alpha < \kappa$ with $x \ntrianglelefteq_\alpha y$.
In fact, if $x \preceq y$ then $x \trianglelefteq_\alpha y$ holds for every $\alpha$, and if $y \preceq x$ then $x \trianglelefteq_\alpha y$ never holds.

\bigskip

Now we introduce reverse mathematics, which is a program in the foundations of mathematics to establish the axiomatic strength of theorems of ordinary mathematics within the context of subsystems of second order arithmetic.
In particular all combinatorial objects considered in this setting are countable.
We refer to \cite{simpson} and \cite{dzhafarovmummert} for a wider introduction.

The usual basis system to develop reverse mathematics is $\RCA$, which consists of the first order axioms of ordered semirings, induction restricted to \SI01 formulas and comprehension restricted to \DE01 predicates.
$\RCA$ roughly corresponds to computable mathematics.

An example of a statement provable in $\RCA$ is Theorem \ref{szpilrajn} restricted to countable posets.
Moreover, the proof is uniform and so multiple applications of Theorem \ref{szpilrajn} are available in $\RCA$.

\begin{theorem}[$\RCA$]\label{uniformszpilrajn}
For each sequence of posets $(P_i, \preceq_i)_{i \in \N}$ there exists a sequence of linearizations $(P_i, \trianglelefteq_i)_{i \in \N}$.
\end{theorem}

\begin{proof}
Let $\langle \cdot, \cdot \rangle \colon \N \times \N \to \N$ be a computable bijection and for each $i \in \N$ fix a listing $(x^i_n)_{n \in \N}$ of the elements of $P_i$.
We define simultaneously for each $i$ a linear order $\trianglelefteq_i$ which extends $\preceq_i$ and is $\preceq_i$ computable.
We proceed by stages.
Suppose we are at stage $s = \langle i,m \rangle$: we want to stipulate the $\lhd_i$ comparabilities of $x^i_m$.
For each $\langle i,n \rangle < s$, if there is $\langle i,n' \rangle < s$ such that $x^i_{n'} \prec_i x^i_m$ and $x^i_n \lhd_i x^i_{n'}$ then we put $x^i_n \lhd_i x^i_m$.
Otherwise we put $x^i_m \lhd_i x^i_n$.

Then for each $i$ each relation $\trianglelefteq_i$ extends the corresponding $\preceq_i$ and for any $x,y \in P_i$ such that $x \ne y$, either $x \lhd_i y$ or $y \lhd_i x$.
Moreover, $\trianglelefteq_i$ is still a poset and it is computable from $\preceq_i$ as required.
\end{proof}

We can also prove that a realization of a poset always exists in $\RCA$ by formalizing an argument of \cite{dushnikmiller}.

\begin{theorem}[$\RCA$]\label{dimensionrca}
For each poset $(P, \preceq)$, there exists a set $\{\trianglelefteq_n : n \in \N\}$ of linearizations of $(P,\preceq)$ which realizes it.
\end{theorem}

\begin{proof}
If $(P, \preceq)$ is a chain, then $\{ \preceq \}$ is a realization.

Assume that $(P, \preceq)$ is not a chain and let $a,b \in P$ be such that $a \mid b$. 
We extend $(P, \preceq)$ to another poset $(P,  \preceq_a^b)$ in the following way: for each $x,y \in P$ we stipulate that $x \preceq_a^b y$ if and only if either $x \preceq y$ or $x \preceq a \wedge b \preceq y$.
Notice that such relation exists in $\RCA$, ${\preceq} \subseteq {\preceq_a^b}$ and $a \preceq_a^b b$.
It is straightforward to check that $(P, \preceq_a^b)$ is a poset.


By Theorem \ref{uniformszpilrajn}, for each $a,b \in P$ such that $a \mid b$, there exists a linearization $\trianglelefteq_a^b$ of $(P, \preceq_a^b)$.
It is clear that if $x \mid y$ then $x \ntrianglelefteq_y^x y$, so that $\{\trianglelefteq_a^b: a,b \in P, a \mid b\}$ realizes $(P, \preceq)$.
\end{proof}

The existence of the dimension of a poset, i.e.\ the existence of a realization of minimal cardinality, appears to require either $\WKL$ (see Remark \ref{WKLdim} below for details) or \PI11-induction (which yields the \SI11 least number principle).
For this reason, when we write $\dim(P,\preceq) \leq n$ in $\RCA$ we mean that there exist $n$ linearizations which realize $P$. 
Similarly, $\dim(P,\preceq) \geq n$ means that $n-1$ linearizations never realize $P$.

We are interested in systems extending $\RCA$.
For instance $\WKL$, one of the well known big five of reverse mathematics, is obtained by adding weak König's lemma (namely the statement asserting that every infinite binary tree has an infinite path).

\begin{theorem}[$\RCA$]\label{chomarsol}
The following are equivalent:
\begin{enumerate}[leftmargin=*,label=(\arabic*)]
    \item $\WKL$;
    \item \label{1} for every $1-1$ functions $f,g : \N \rightarrow \N$ with disjoint ranges there exists a set which separates the ranges;
    \item \label{2} every acyclic relation can be extended to a partial order;
    \item \label{3} for every partial order $(P, \preceq)$ and sets $I, F \subseteq P$ such that $\forall x \in I\, \forall y \in F (y \npreceq x)$ there exists a downward closed set $B \subseteq P$ such that $I \subseteq B$ and $B \cap F = \emptyset$.
\end{enumerate}
\end{theorem}

The equivalence of $\WKL$ with \ref{1} is a classical result, the equivalence with \ref{2} was proved in \cite{cholakmarconesolomon}, while the equivalence with \ref{3} was proved in \cite{frittaionmarcone}.

The statement \ref{2} in Theorem \ref{chomarsol} self-strengthens to its infinite version.

\begin{corollary}[$\WKL$]\label{uniformchomarsol}
If $(R_i)_{i \in \N}$ are acyclic relations on sets $(P_i)_{i \in \N}$ respectively, then each $R_i$ can be extended to a partial order.
\end{corollary}

\begin{proof}
Let $P$ and $R$ be the disjoint unions respectively of $(P_i)_{i \in \N}$ and $(R_i)_{i \in \N}$.
Then $R$ is an acyclic relation on the set $P$ and we can apply \ref{3} of Theorem \ref{chomarsol} to get a partial order $(P, \preceq)$ which extends $R$.
Each restriction of $\preceq$ to $P_i$ is a partial order extending $R_i$.
\end{proof}

Another axiomatic system we are interested in is $\I \SI02$, the strengthening of $\RCA$ with induction for all $\SI02$ formulas.
This system is less popular than the big five, but theorems equivalent to it have been studied for instance in \cite{hirst, gurahirstmummert}.
We recall here a basic fact (see \cite[Exercise II.3.13]{simpson}).

\begin{theorem}[$\RCA$]\label{isigma}
The following are equivalent:
\begin{enumerate}[leftmargin=*,label=(\arabic*)]
    \item $\I \SI02$;
    \item bounded $\SI02$ comprehension: for each $\SI02$ formula $\varphi$ with $X$ not free
    $$\forall z\, \exists X\, \forall x (x \in X \leftrightarrow x < z \wedge \varphi(x)).$$
\end{enumerate}
\end{theorem}

We show that for each $n \in \N$ there exists a poset of dimension exactly $n$.

\begin{definition}\label{fn}
Let $n>1$ and let $F_n = \{a_i,b_i: i < n\}$.
Equip $F_n$ with the partial order ${\prec} = \{(a_i, b_j) : i \ne j\}$.

\medskip

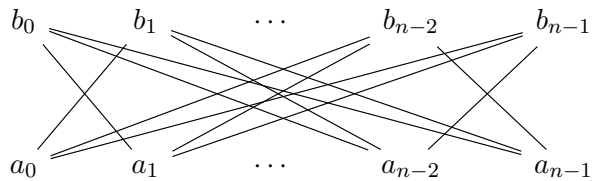
\adjustbox{scale=1,center}{
\begin{tikzcd}
	{b_0} & {b_1} & \cdots & {b_{n-2}} & {b_{n-1}} \\
	\\
	{a_0} & {a_1} & \cdots & {a_{n-2}} & {a_{n-1}}
	\arrow[no head, from=3-1, to=1-2]
	\arrow[no head, from=3-1, to=1-4]
	\arrow[no head, from=3-1, to=1-5]
	\arrow[no head, from=3-2, to=1-1]
	\arrow[no head, from=3-2, to=1-4]
	\arrow[no head, from=3-2, to=1-5]
	\arrow[no head, from=3-4, to=1-5]
	\arrow[no head, from=3-4, to=1-2]
	\arrow[no head, from=3-4, to=1-1]
	\arrow[no head, from=3-5, to=1-4]
	\arrow[no head, from=3-5, to=1-2]
	\arrow[no head, from=3-5, to=1-1]
\end{tikzcd}}
\medskip
For simplicity, we will refer to this poset by $F_n$.
\end{definition}

\begin{theorem}[$\RCA$]\label{dimfn}
The poset $F_n$ has dimension $n$.
\end{theorem}

\begin{proof}
To prove that $\dim(F_n) \ge n$ it suffices to show that for every linearization $\trianglelefteq$ of $F_n$ there exists at most one $k<n$ such that $b_k \trianglelefteq a_k$.
In fact, if $b_k \trianglelefteq a_k$ and $b_{k'} \trianglelefteq a_{k'}$ for $k \ne k'$, then 
$$b_{k'} \trianglelefteq a_{k'} \trianglelefteq b_k \trianglelefteq a_k \trianglelefteq b_{k'},$$
and by antisymmetry of $\trianglelefteq$ we obtain $a_k = a_{k'} = b_k = b_{k'}$, a contradiction.

For the converse it suffices to define a realization consisting of exactly $n$ linearizations.
For each $i < n$ let $\trianglelefteq_i$ be the following linearization of $F_n$:
\begin{multline*}
    a_0 \trianglelefteq_i \ldots \trianglelefteq_i a_{i-1} \trianglelefteq_i a_{i+1} \trianglelefteq_i \ldots \trianglelefteq_i a_{n-1} \trianglelefteq_i b_i \trianglelefteq_i \\
    \trianglelefteq_i a_i \trianglelefteq_i b_{n-1} \trianglelefteq_i \ldots b_{i+1} \trianglelefteq_i b_{i-1} \trianglelefteq_i \ldots \trianglelefteq_i b_0.    
\end{multline*}
Each $\trianglelefteq_i$ exists in $\RCA$ by \SI00-comprehension.
We claim that $\{\trianglelefteq_0, \ldots, \trianglelefteq_{n-1}\}$ realizes $F_n$.
If $x \mid y$ we need to find some $i$ such that $x \ntriangleleft_i y$.
If $x=a_j$ and $y=a_k$ for $j \ne k$ it holds that $a_j \ntriangleleft_j a_k$.
If $x=b_j$ and $y = b_k$ for $j \ne k$ it holds that $b_j \ntriangleleft_k b_k$.
If $x=b_j$ and $y = a_j$ it holds that $b_j \ntriangleleft_k a_j$ for any $k \ne j$.
Finally, if $x=a_j$ and $y = b_j$ it holds that $a_j \ntriangleleft_j b_j$.

We conclude that $\{\trianglelefteq_0, \ldots, \trianglelefteq_{n-1}\}$ realizes $F_n$ and so $\dim(F_n) = n$.
\end{proof}

If $\{\trianglelefteq_0, \ldots, \trianglelefteq_{m-1}\}$ realizes $F_n$ for some $m > n$ then a priori the pair $(b_i,a_i)$ may belong to more than one linearization.
On the other hand, the proof of Theorem \ref{dimfn} suggests that the pair $(b_i,a_i)$ occurs in exactly one linearization of any realization of $F_n$ of size $n$.
Indeed this is the case.

\begin{lemma}[$\RCA$]\label{lemmino}
For each set of linearizations $\{\trianglelefteq_0, \ldots, \trianglelefteq_{n-1}\}$ that realizes $F_n$ and for each $k < n$ there exists exactly one $i < n$ such that $b_k \trianglelefteq_i a_k$.
\end{lemma}

\begin{proof}
If $n = 2$ the result is trivial so we may assume $n > 2$.
For each $k < n$, the existence of $i$ is trivial since $a_k \mid b_k$.

To prove uniqueness suppose, without loss of generality, that $b_{n-1} \trianglelefteq_{n-1} a_{n-1}$ and $b_{n-1} \trianglelefteq_{n-2} a_{n-1}$. 
Then, by the first argument in the proof of Theorem \ref{dimfn}, the set of restrictions of $\{\trianglelefteq_0, \ldots, \trianglelefteq_{n-3}\}$ to $F_{n-1}$ realizes $F_{n-1}$, against Theorem \ref{dimfn} itself.
\end{proof}

$\WKL$ is equivalent to the reduction of the dimension of a poset to its finite subposets.

\begin{theorem}\label{th:dimension}
For each $k \ge 3$, $\RCA$ proves the equivalence between:
\begin{enumerate}[leftmargin=*,label=(\arabic*)]
    \item \label{dim1} $\WKL$;
    \item \label{dim2} for every $n$ and every poset $P$, if $\dim(Q,\preceq) \leq n$ for every finite subposet $Q$ of $P$, then $\dim(P,\preceq) \leq n$;
    \item \label{dim3} for every poset $P$, if $\dim(Q,\preceq) \leq k$ for every finite subposet $Q$ of $P$, then $\dim(P,\preceq) \leq k$.
\end{enumerate}
\end{theorem}

\begin{proof}
The implication from \ref{dim1} to \ref{dim2} is a straightforward application of $\WKL$ and the details are left to the reader, while  \ref{dim3} is an easy consequence of $\ref{dim2}$.

We now deal with the implication from \ref{dim3} to \ref{dim1} and exploit item \ref{1} of Theorem \ref{chomarsol}.
Let $f,g$ be 1-1 functions with disjoint ranges: we want to show that there exists a set $A$ such that $\ran(f) \subseteq A$ and $A \cap \ran(g) = \emptyset$. 
We construct a poset $(P^k,\preceq)$ (recall that $k \geq 3$ is fixed) that will be used for many reversals in Section \ref{sec:chains}.
Let 
$$P^k=\{x^i,y^i:i \in \N\} \cup \{c_j^r,d_j^r: j < k-1, r \in \N\} \cup \{p_j^s,q_j^s: j < k-1, s \in \N\}.$$

To partially order $P^k$ we start with a level function $\ell : P \rightarrow \N$ defined by
$$\ell (z) = \begin{cases}
i & \text{if $z=x^i,y^i$}; \\
f(r) & \text{if $z=c_j^r,d_j^r$}; \\
g(s) & \text{if $z=p_j^s,q_j^s$}.
\end{cases}
$$
Notice that $\RCA$ suffices to prove the existence of $\ell$.
For every $n$ let $P^k_n = \{z \in P^k: \ell (z)=n\}$: we say that $P^k_n$ is a level of $P^k$.
Since $f$ is 1-1, $c^r_j$ and $d^{r'}_h$ belong to the same level if and only if $r = r'$, and the same holds for $p^s_j$ and $q^{s'}_h$.
Furthermore we have $\ell (c^r_j) \neq \ell (p^s_h)$ for all $r$, $s$, $j$ and $h$.

We define the strict partial order $\prec$ by setting for each $u, v \in P$:
\begin{itemize}
    \item if $\ell(u) < \ell(v)$ then $u \prec v$,
    \item if $\ell(u) = \ell(v) = m$ then $u \prec v$ if and only if one of the following six alternatives holds:
\begin{alignat*}{2}
    u =x^m,\, v = d^r_j,\, f(r) & =m & u & =p^s_j,\, v=x^m,\, g(s)=m\\
    u =c^r_j,\, v =y^m,\, f(r) & =m & u & =y^m,\, v=q^s_j,\, g(s)=m\\
    u =c^r_j,\, v=d^r_{j'},\, j \ne j',\, f(r) & =m \quad\qquad & u & =p^s_j,\, v=q^s_{j'},\, j \ne j',\, g(s)=m.
\end{alignat*}
\end{itemize}

To check that $(P^k,\preceq)$ is a partial order, notice that asymmetry is trivial, while transitivity is immediate when the elements belong to different levels (because the levels are linearly ordered by $\preceq$ as the natural numbers) and follows easily in the other case because no level has chains of length $3$.

If $m \notin \ran(f) \cup \ran(g)$ then $P^k_m = \{x^m, y^m\}$ is an antichain.
If $f(r)=m$ the level $P^k_m$ is a copy of $F_k$ where $\{x^m,c^r_0,\ldots,c^r_{k-2}\}$ is the set of the $a_i$'s while $\{y^m,d^r_0,\ldots,d^r_{k-2}\}$ is the set of the $b_i$'s (recall the notation of Definition \ref{fn}).
The figure below illustrates level $m$:

\adjustbox{scale=1,center}{
\begin{tikzcd}
	{y^m} & {d^r_0} & \cdots & {d^r_{k-3}} & {d^r_{k-2}} \\
	\\
	{x^m} & {c^r_0} & \cdots & {c^r_{k-3}} & {c^r_{k-2}}
	\arrow[no head, from=3-1, to=1-2]
	\arrow[no head, from=3-1, to=1-4]
	\arrow[no head, from=3-1, to=1-5]
	\arrow[no head, from=3-2, to=1-1]
	\arrow[no head, from=3-2, to=1-4]
	\arrow[no head, from=3-2, to=1-5]
	\arrow[no head, from=3-4, to=1-5]
	\arrow[no head, from=3-4, to=1-2]
	\arrow[no head, from=3-4, to=1-1]
	\arrow[no head, from=3-5, to=1-4]
	\arrow[no head, from=3-5, to=1-2]
	\arrow[no head, from=3-5, to=1-1]
\end{tikzcd}}

Similarly, if $g(s)=m$ the level $P^k_m$ is a copy of $F_k$ where $\{y^m,p^s_0,\ldots,p^s_{k-2}\}$ is the set of the $a_i$'s while $\{x^m,q^s_0,\ldots,q^s_{k-2}\}$ is the set of the $b_i$'s.
The figure below illustrates level $m$ in this case:

\adjustbox{scale=1,center}{
\begin{tikzcd}
	{x^m} & {q^r_0} & \cdots & {q^r_{k-3}} & {q^r_{k-2}} \\
	\\
	{y^m} & {p^r_0} & \cdots & {p^r_{k-3}} & {p^r_{k-2}}
	\arrow[no head, from=3-1, to=1-2]
	\arrow[no head, from=3-1, to=1-4]
	\arrow[no head, from=3-1, to=1-5]
	\arrow[no head, from=3-2, to=1-1]
	\arrow[no head, from=3-2, to=1-4]
	\arrow[no head, from=3-2, to=1-5]
	\arrow[no head, from=3-4, to=1-5]
	\arrow[no head, from=3-4, to=1-2]
	\arrow[no head, from=3-4, to=1-1]
	\arrow[no head, from=3-5, to=1-4]
	\arrow[no head, from=3-5, to=1-2]
	\arrow[no head, from=3-5, to=1-1]
\end{tikzcd}}

Notice that in $\RCA$ we can realize each level with $k$ linearizations and the levels are linearly ordered.
If $Q$ is a finite subset of $P^k$ then $Q$ intersects finitely many $P^k_m$ and, using bounded \SI01-comprehension which is available in $\RCA$, we know how to realize each of these levels.
Hence $\dim(Q,\preceq) \leq k$ and by \ref{dim3} we have $\dim(P^k,\preceq) \le k$.
Let $\{\trianglelefteq^*_i : i < k\}$ be a set of linearizations which realizes $(P^k,\preceq)$.
Notice that for each $i < k$, the restriction of $\trianglelefteq^*_i$ to the level $P^k_m$ is a linearization of $(P^k_m, \preceq)$ and the set of these restrictions realizes $(P^k_m, \preceq)$.
We already noticed that $P^k_m$ is either an antichain of two elements or a copy of $F_k$.
By Lemma \ref{lemmino} if $m \in \ran(f)$ then $|\{i < k : y^m \trianglelefteq^*_i x^m\}| =1$ while if $m \in \ran(g)$ then $|\{i < k : y^m \trianglelefteq^*_i x^m\}| =k-1$.
If $m \notin \ran(f) \cup \ran(g)$ we only know that $1 \le |\{i < k : y^m \trianglelefteq^*_i x^m\}| \le k-1$.
If we let 
$$A = \{m : |\{i < k : y^m \trianglelefteq^*_i x^m\}| =1 \}$$
we have $\ran(f) \subseteq A$ and $A \cap \ran(g) = \emptyset$, as desired.
\end{proof}

\begin{remark}\label{WKLdim}
By Theorem \ref{th:dimension}, in $\WKL$ the formula $\dim(P,\preceq) \leq n$ is equivalent to a \PI01 formula. 
Hence the least $m$ such that $\dim(P,\preceq) \leq m$ exists and we can define the dimension of a poset.
\end{remark}

\section{Bounding theorems}\label{sec:boundth}

Several theorems have been proved about the dimension of posets: for example in \cite{fishburn} there are several statements that bound the dimension of a poset in terms of other quantities, such as its width.
We are interested in statements that give an upper bound to the dimension of a poset in terms of the dimension of its subposets, obtained by removing one or more points:
\begin{itemize}    
    \item $\DBp$: for each poset $(P, \preceq)$ and each $x_0 \in P$, 
    $$\dim(P,\preceq) \le \dim(P \setminus \{x_0\},\preceq)+1.$$
    \item $\DBin$: for each poset $(P, \preceq)$ and each set of pairwise incomparable chains $C_i \subseteq P$ for $i < n$
    $$\textstyle\dim(P, \preceq) \le \dim(P \setminus \bigcup_{i < n}C_i, \preceq) +\max\{2,n\}.$$
\end{itemize}

$\DBp$, $\DBio$ and $\DBit$ are proved in \cite{hiraguti}, while for $n \ge 3$ $\DBin$ is a new (to the best of our knowledge) natural extension of the previous results.

When formalizing these statements in $\RCA$ by $\dim (P) \leq \dim(Q) + k$ we mean \lq\lq for every $m$ and every realization of $Q$ with $m$ linearizations there is a realization of $P$ with $m+k$ linearizations\rq\rq.


We show that each $\DBin$ is provable in, and in fact equivalent to, $\WKL$ in Section \ref{sec:chains}.
In Section \ref{sec:point} we deal with $\DBp$, showing how to prove it either in $\WKL$ or using $\I\SI02$.


Starting from the posets $F_n$, it is fairly easy to construct posets that show that the bounds provided by $\DBin$ are sharp.
In the examples we highlight the crucial comparabilities.

\begin{example}
Fix $n \ge 2$ and consider $F_{n+2}$ and the chain $C= \{a_n,b_{n+1}\}$. 
The poset $F_{n+2} \setminus C$ is a copy of $F_{n+1}$ plus the relation $a_n \prec b_n$ (it suffices to rename $a_{n+1}$ as $a_n$).
In the figure we show the poset highlighting the elements and the comparability that we remove.

\smallskip
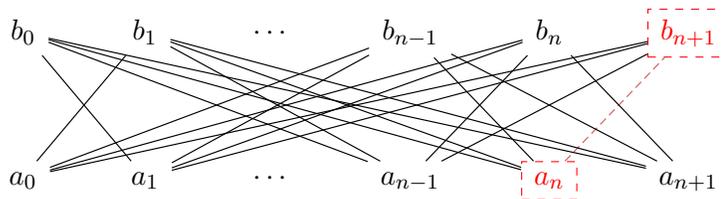
\adjustbox{scale=1,center}{
\begin{tikzcd}
	{b_0} & {b_1} & \cdots & {b_{n-1}} & {b_n} &  |[draw, dashed]| {b_{n+1}} \\
	\\
	{a_0} & {a_1} & \cdots & {a_{n-1}} &  |[draw, dashed]| {a_{n}} & {a_{n+1}}
	\arrow[no head, from=3-1, to=1-2]
	\arrow[no head, from=3-1, to=1-4]
	\arrow[no head, from=3-1, to=1-5]
	\arrow[no head, from=3-1, to=1-6]
	\arrow[no head, from=3-2, to=1-1]
	\arrow[no head, from=3-2, to=1-4]
	\arrow[no head, from=3-2, to=1-5]
	\arrow[no head, from=3-2, to=1-6]
	\arrow[no head, from=3-4, to=1-1]
	\arrow[no head, from=3-4, to=1-2]
	\arrow[no head, from=3-4, to=1-5]
	\arrow[no head, from=3-4, to=1-6]
	\arrow[no head, from=3-5, to=1-1]
	\arrow[no head, from=3-5, to=1-2]
	\arrow[no head, from=3-5, to=1-4]
	\arrow[dashed, no head, from=3-5, to=1-6]
	\arrow[no head, from=3-6, to=1-5]
	\arrow[no head, from=3-6, to=1-4]
	\arrow[no head, from=3-6, to=1-2]
	\arrow[no head, from=3-6, to=1-1]
\end{tikzcd}}
\smallskip

If $\trianglelefteq_0, \ldots, \trianglelefteq_n$ are the $n+1$ linearizations constructed in the proof of Theorem \ref{dimfn} to realize $F_{n+1}$, it is immediate to see that $\trianglelefteq_0, \ldots, \trianglelefteq_{n-1}$ suffice to realize $F_{n+2} \setminus C$.
Therefore $\dim(F_{n+2} \setminus C) \le n$ while $\dim(F_{n+2} ) = n+2$ by Theorem \ref{dimfn}.
Therefore the bound in $\DBio$ is sharp.    
\end{example}

\begin{example}\label{optdbc2}
Consider again $F_{n+2}$ for $n \ge 2$, and let $C_0 = \{a_{n+1},b_n\}$ and $C_1 = \{a_n,b_{n+1}\}$.
$C_0$ and $C_1$ are incomparable chains, $F_{n+2} \setminus (C_0 \cup C_1)$ is exactly $F_n$, so that we have $\dim(F_{n+2} \setminus (C_0 \cup C_1)) = n$ and the bound in $\DBit$ is sharp.
\end{example}

We now generalize Example \ref{optdbc2} to $\DBin$ for $n > 2$.
We need to adjust the construction, as in $F_n$ there are no more than two incomparable chains of the form $\{a_i,b_j\}$.

\begin{example}
Fix $n \ge 3$.
Consider the poset $F_{n+2}$ and the $n$ incomparable chains $C_i = \{b_i\}$ for $i < n$.
Let $C = \bigcup_{i < n} C_i$: we claim that $\dim(F_{n+2} \setminus C) = 2$.
In the figure we show the poset after we removed the elements.
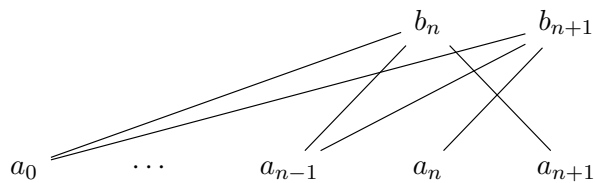
\[\begin{tikzcd}
	&&& {b_n} & {b_{n+1}} \\
	\\
	{a_0} & \cdots & {a_{n-1}} & {a_n} & {a_{n+1}}
	\arrow[no head, from=3-1, to=1-4]
	\arrow[no head, from=3-1, to=1-5]
	\arrow[no head, from=3-3, to=1-4]
	\arrow[no head, from=3-3, to=1-5]
	\arrow[no head, from=3-4, to=1-5]
	\arrow[no head, from=3-5, to=1-4]
\end{tikzcd}\]
To prove that $\dim(F_{n+2} \setminus C) \le 2$ we explicitly define two linearizations $\trianglelefteq_0$ and $\trianglelefteq_1$:
\begin{gather*}
    a_{n+1} \trianglelefteq_0 a_0 \trianglelefteq_0 \ldots \trianglelefteq_0 a_{n-1} \trianglelefteq_0 b_n \trianglelefteq_0 a_n \trianglelefteq_0 b_{n+1};\\
    a_n \trianglelefteq_1 a_{n-1} \trianglelefteq_1 \ldots \trianglelefteq_1 a_0 \trianglelefteq_1 b_{n+1} \trianglelefteq_1 a_{n+1} \trianglelefteq_1 b_n.
\end{gather*}
It is routine to check that $\trianglelefteq_0$ and $\trianglelefteq_1$ realize $F_{n+2} \setminus C$. 
\end{example}

One may wonder if the requirement in $\DBin$ ($n \ge 2$) that the chains are incomparable is necessary. 
The following example shows that this is the case.
For simplicity, we deal with the case $n = 2$ though the construction can be adapted to any $n$.

\begin{example}\label{f5}
Consider $F_5$ and the chains $C_0 = \{a_0,b_1\}$ and $C_1 = \{a_2,b_3\}$ which are comparable since $a_0 \preceq b_3$ and $a_2 \preceq b_1$.
Then $F_5 \setminus (C_0 \cup C_1)$ is $F_3$ with two new comparabilities between $a$'s and $b$'s highlighted in the figure.

\smallskip
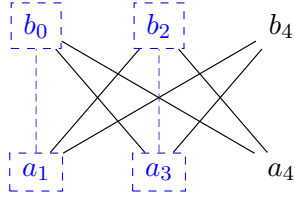
\adjustbox{scale=1,center}{
\begin{tikzcd}
	 |[draw=blue, dashed, text=blue]| {b_{0}} &  |[draw=blue, dashed, text=blue]| {b_{2}} & {b_4} \\
	\\
	 |[draw=blue, dashed, text=blue]| {a_{1}} &  |[draw=blue, dashed, text=blue]| {a_3} & {a_4}
	\arrow[no head, from=3-2, to=1-3]
	\arrow[no head, from=3-3, to=1-2]
	\arrow[no head, from=3-1, to=1-2]
	\arrow[no head, from=3-1, to=1-3]
	\arrow[color={rgb,255:red,92;green,92;blue,214}, dashed, no head, from=3-1, to=1-1]
	\arrow[color={rgb,255:red,92;green,92;blue,214}, dashed, no head, from=3-2, to=1-2]
	\arrow[no head, from=3-3, to=1-1]
	\arrow[no head, from=3-2, to=1-1]
\end{tikzcd}}
\smallskip

The linearizations 
\begin{gather*}
a_1 \trianglelefteq_0 a_3 \trianglelefteq_0 b_4 \trianglelefteq_0 a_4 \trianglelefteq_0 b_0 \trianglelefteq_0 b_2; \\  
a_4 \trianglelefteq_1 a_3 \trianglelefteq_1 a_1 \trianglelefteq_1 b_2 \trianglelefteq_1 b_0 \trianglelefteq_1 b_4
\end{gather*}
realize $F_5 \setminus (C_0 \cup C_1)$.
We conclude that the dimension is exactly $2$ and consequently $\dim(F_5) = 5 \nleq 4 = \dim(F_5 \setminus (C_0 \cup C_1)) + 2$.
\end{example}

Despite Example \ref{f5}, we may still state a bound without the requirement of incomparability of the chains. 
\begin{itemize}    
    \item $\DBcn$: for each poset $(P, \preceq)$ and each family of chains $C_i \subseteq P$ for $i < n$, 
    $$\textstyle\dim(P, \preceq) \le \dim(P \setminus \bigcup_{i < n}C_i, \preceq) +2n.$$
\end{itemize}

We show that this bound is sharp.
As with Example \ref{f5}, for simplicity we deal with the case $n = 2$.

\begin{example}\label{f6}
Consider $F_6$, $C_0 = \{a_0,b_1\}$ and $C_1 = \{a_2,b_3\}$. 
Then $F_6 \setminus (C_0 \cup C_1)$ is $F_4$ with two comparabilities between the $a's$ and the $b's$ highlighted in the figure.

\smallskip
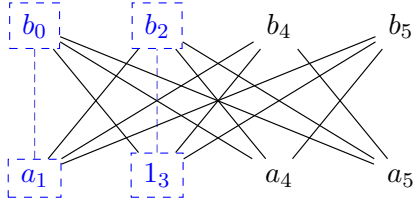
\adjustbox{scale=1,center}{
\begin{tikzcd}
	|[draw=blue, dashed, text=blue]| {b_{0}} & |[draw=blue, dashed, text=blue]| {b_{2}} & {b_4} & {b_5} \\
	\\
	|[draw=blue, dashed, text=blue]| {a_{1}} & |[draw=blue, dashed, text=blue]| {1_{3}} & {a_4} & {a_5}
	\arrow[no head, from=3-2, to=1-3]
	\arrow[no head, from=3-3, to=1-2]
	\arrow[no head, from=3-1, to=1-2]
	\arrow[no head, from=3-1, to=1-3]
	\arrow[color={rgb,255:red,92;green,92;blue,214}, dashed, no head, from=3-1, to=1-1]
	\arrow[color={rgb,255:red,92;green,92;blue,214}, dashed, no head, from=3-2, to=1-2]
	\arrow[no head, from=3-3, to=1-1]
	\arrow[no head, from=3-2, to=1-1]
	\arrow[no head, from=3-1, to=1-4]
	\arrow[no head, from=3-2, to=1-4]
	\arrow[no head, from=3-4, to=1-3]
	\arrow[no head, from=3-4, to=1-2]
	\arrow[no head, from=3-4, to=1-1]
	\arrow[no head, from=3-3, to=1-4]
\end{tikzcd}}
\smallskip

The linearizations 
\begin{gather*}
a_5 \trianglelefteq_0 a_3 \trianglelefteq_0 a_1 \trianglelefteq_0 b_4 \trianglelefteq_0 a_4 \trianglelefteq_0 b_0 \trianglelefteq_0 b_2 \trianglelefteq_0 b_5; \\
a_4 \trianglelefteq_1 a_1 \trianglelefteq_1 a_3 \trianglelefteq_1 b_5 \trianglelefteq_1 a_5 \trianglelefteq_1 b_2 \trianglelefteq_1 b_0 \trianglelefteq_1 b_4
\end{gather*}
realize $F_6 \setminus (C_0 \cup C_1)$ so that $\dim(F_6) = 6 = \dim(F_6 \setminus (C_1 \cup C_2)) + 4$.
\end{example}

A simple proof of $\forall n\, \DBcn$ can be obtained by applying repeatedly $\DBio$.
However this proof cannot be formalized in $\WKL$, and in Section \ref{sec:chains} we give a direct proof.

\section{$\DBin$ and $\DBcn$}\label{sec:chains}

In this section we deal with $\DBin$ and $\DBcn$.
We show that for each $n$, each of these dimension bounds is equivalent to $\WKL$.
The following notion is the basic tool for proving these statements.

\begin{definition}
Let $(P,\preceq)$ be a poset and let $C^0, C^1 \subseteq P$ be incomparable chains. 
We say that a linearization $\trianglelefteq$ \textit{puts $C^0$ at the bottom and $C^1$ at the top of $P$} if for each $x \in P$, $c_0 \in C^0$ and  $c_1 \in C^1$, if $x \mid c_0$ and $x \mid c_1$ then $c_0 \trianglelefteq x \trianglelefteq c_1$.
\end{definition}

\begin{lemma}[$\WKL$]\label{chains}
Let $(P_i,\preceq_i)_{i \in \N}$ be a sequence of partial orders and let $(C_i^j)_{i \in \N, j < 2}$ be such that for each $i$ $C_i^0$ and $C_i^1$ are incomparable chains in $P_i$. 
There exists a sequence $(P_i, \trianglelefteq_i)_{i \in \N}$ of linearizations of $(P_i,\preceq_i)_{i \in \N}$ such that each $\trianglelefteq_i$ puts $C_i^0$ at the bottom and $C_i^1$ at the top of $P_i$.
\end{lemma}

\begin{proof}
For each $i \in \N$ define $R_i^0$ and $R_i^1$ over $P_i$ as follows:
\begin{align*}
x R_i^0 y & \text{ if and only if } x \mid y \land x \in C_i^0 \land y \notin C_i^0;\\
x R_i^1 y & \text{ if and only if } x \mid y \land x \notin C_i^1 \land y \in C_i^1.
\end{align*}

We claim that each $R_i = {\preceq_i} \cup R_i^0 \cup R_i^1$ is acyclic.
Assume that 
$$x_0 R_i x_1 R_i \ldots R_i x_{m-1} R_i x_0.$$
Since $\preceq_i$ is transitive, we may assume that it does not occur twice consecutively.
It is immediate that also each of $R_i^0$ or $R_i^1$ cannot occur twice consecutively.
Since the chains are incomparable we cannot have $x_j R_i^1 x_{j+1} R_i^0 x_{j+2}$ and $x_j R_i^1 x_{j+1} {\preceq_i} x_{j+2} R_i^0 x_{j+3}$ as well.

If $R_i^1$ occurs in the cycle, up to renaming its elements, we may assume that $x_0 R_i^1 x_1 \preceq_i x_2 \ldots$ and then $R_i^0$ does not occur. 
Therefore $\preceq_i$ and $R_i^1$ alternate in the cycle and this entails that $m$ is even.
Hence, the cycle has the form
$$x_0 R_i^1 x_1 {\preceq_i} \ldots R_i^1 x_{m-1} {\preceq_i} x_0$$
For each $j < \frac{m}{2}$ we have that $x_{2j} \notin C_i^1$ while $x_{2j+1} \in C_i^1$. 
We claim that for each $j < \frac{m}{2}-1$, $x_{2j+1} \preceq_i x_{2j+3}$. 
Indeed $x_{2j+1} \preceq_i x_{2j+2}$ and $x_{2j+2} \mid x_{2j+3}$. 
Since $x_{2j+1}, x_{2j+3} \in C_i^1$ it must be
$x_{2j+1} \preceq_i x_{2j+3}$.
Thus $x_1 \preceq_i x_{m-1} \preceq_i x_0$, contradicting $x_0 \mid x_1$.

If $R_i^1$ does not occur in the cycle, then $R_i^0$ does and an analogous argument leads to a contradiction, completing the proof of the claim.

By Corollary \ref{uniformchomarsol}, which is the step where we use $\WKL$, and by Theorem \ref{uniformszpilrajn} the relations $(R_i)_{i \in \N}$ can be extended to linear orders $(\trianglelefteq_i)_{i \in \N}$ which have the prescribed properties.
\end{proof}

Notice that we include the case in which one of the chains is empty.

Lemma \ref{chains} can be reversed.
For simplicity, we state the result only in the case of a single partial order.

\begin{theorem}[$\RCA$]
The following are equivalent:
\begin{enumerate}[leftmargin=*,label=(\arabic*)] 
    \item \label{4.1.1} $\WKL$;
    \item \label{4.1.2} if $(P,\preceq)$ is a poset and $C_0, C_1 \subseteq P$ are incomparable chains, there exists a linearization $(P, \trianglelefteq)$ of $(P, \preceq)$ that puts $C_0$ at the bottom and $C_1$ at the top of $P$.
\end{enumerate}
\end{theorem}

\begin{proof}
Lemma \ref{chains} proves that \ref{4.1.1} implies \ref{4.1.2}.

For the converse let $f,g$ be 1-1 functions with disjoint ranges and let $P=\{z\} \cup \{x_n,a_n,b_n : n \in \N\}$.
We define a partial order $\preceq$ on $P$ as follows:
\begin{itemize}
    \item $a_i \preceq a_j$ if and only if $i < j$;
    \item $b_i \preceq b_j$ if and only if $j < i$;
    \item $x_n \preceq a_i$ if and only if $\exists j < i+1\, f(j) = n$;
    \item $b_i \preceq x_n$ if and only if $\exists j < i+1\, g(j) = n$.
\end{itemize}

Let $C_0 = \{a_n : n \in \N\}$ and $C_1 = \{b_n : n \in \N\}$ which are incomparable chains.
By \ref{4.1.2} let $\trianglelefteq$ be a linearization of $\preceq$ that puts $C_0$ at the bottom and $C_1$ at the top of $P$.
The set $\{n : x_n \trianglelefteq z\}$ separates $\ran(f)$ and $\ran(g)$.  
\end{proof}

We now deal with the bound $\DBin$ and show that it is provable in $\WKL$.
As we mentioned in Section \ref{sec:boundth}, we only deal with posets with finite dimension.

\begin{theorem}[$\WKL$]\label{hira-cn}
$\forall n > 1 \, \DBin$.
\end{theorem}

\begin{proof}
Let $(P, \preceq)$ be a poset, fix $n > 1$ and a set of pairwise incomparable chains $C_i \subseteq P$ for $i < n$.
Let $C = \bigcup_{i < n} C_i$.
Our goal is to prove that $\dim(P, \preceq) \le \dim(P \setminus C, \preceq) + n$.

Fix $m$ and a set of linearizations $\{\trianglelefteq_0, \ldots, \trianglelefteq_{m-1}\}$ realizing $P \setminus C$. 
We construct a set of $m+n$ linear orders which realizes $P$.

For each $i < m$ let ${\preceq_i^*} = {\preceq} \cup {\trianglelefteq_i}$.
We claim that each $\preceq_i^*$ is acyclic.
Towards a contradiction, let $x_0, \ldots, x_{\ell-1}$ be such that
$$x_0 \preceq_i^* x_1 \preceq_i^* \ldots \preceq_i^* x_{\ell-1} \preceq_i^* x_0.$$
Since both $\preceq$ and $\trianglelefteq_i$ are transitive relations we may assume that $\ell$ is even
$$x_0 \preceq x_1 \trianglelefteq_i \ldots \preceq x_{\ell-1} \trianglelefteq_i x_0.$$
Since $\trianglelefteq_i$ is a relation over $P \setminus C$, it follows that for each $j < \ell$ $x_j \in P \setminus C$.
By the fact that $\trianglelefteq_i$ extends $\preceq$ on its domain, we may replace each occurrence of $\preceq$ with $\trianglelefteq_i$ obtaining a cycle with respect to $\trianglelefteq_i$, a contradiction. 

By Corollary \ref{uniformchomarsol} and Theorem \ref{uniformszpilrajn} each $\preceq_i^*$ can be extended to a linear order $\trianglelefteq_i^*$.

Next we build $n$ further linearizations to deal with the chains $C_i$ for $i < n$.
We apply Lemma \ref{chains} to the $n$ pairs of chains $C_{[j]_n}, C_{[j+1]_n}$ for $j < n$, where $[j]_n$ is the residue class of $j$ modulo $n$.
We obtain a sequence of linear orders $(\trianglelefteq^*_{m+j})_{j < n}$, each extending $(P, \preceq)$ and such that $\trianglelefteq^*_{m+j}$ puts $C_{[j]_n}$ at the bottom and $C_{[j+1]_n}$ at the top of $P$.

We are left to prove that the set $\{\trianglelefteq^*_0, \ldots, \trianglelefteq^*_{m+n-1}\}$ realizes $(P,\preceq)$.
As we noticed in Section \ref{sec:preliminaries} it suffices to prove that for each $x,y \in P$ such that $x \mid y$, there exists $i < m+n$ such that $x \ntrianglelefteq^*_i y$:
\begin{enumerate}[leftmargin=*,label=(\arabic*)]
    \item if $x \in C_i$ then $\trianglelefteq^*_{m + [i-1]_n}$, which puts $C_i$ at the top of $P$, works;
    \item if $y \in C_i$ then $\trianglelefteq^*_{m + [i]_n}$, which puts $C_i$ at the bottom of $P$, works;
    \item if $x,y \in P \setminus C$ then recall that $\{\trianglelefteq_0, \ldots, \trianglelefteq_{m-1}\}$ realizes $P \setminus C$; hence there exists $i < n$ such that $x \ntrianglelefteq_i y$ and consequently $\trianglelefteq^*_i$, which extends $\trianglelefteq_i$, works.
\end{enumerate}

Therefore $\dim(P, \preceq) \le \dim(P \setminus C, \preceq) + n$.    
\end{proof}

Provability of $\DBio$ in $\WKL$ now follows immediately.

\begin{corollary}[$\WKL$]\label{hira-d}
$\DBio$.
\end{corollary}

\begin{proof}
Let $(P, \preceq)$ be a poset and $C \subseteq P$ a chain.
Then, applying Theorem \ref{hira-cn} with $n = 2$, $C_0 = C$ and $C_1 = \emptyset$, we have $\dim(P,\preceq) \le \dim(P \setminus C, \preceq)+2$.
\end{proof}

The last bound we are left to prove is $\DBcn$. 
As $\WKL$ has limited induction, we cannot carry out the straightforward inductive proof applying $\DBio$ (which coincides with $\DBc1$) repeatedly.

\begin{theorem}[$\WKL$]\label{comparable chains}
$\forall n \, \DBcn$.
\end{theorem}

\begin{proof}
Let $(P, \preceq)$ be a poset.
For $i < n$ let $C_i \subseteq P$ be chains and $C = \bigcup_{i < n} C_i$.
We need to prove that $\dim(P, \preceq) \le \dim(P \setminus C, \preceq) + 2n$.

Fix $m$ and a set of linearizations $\{\trianglelefteq_0, \ldots, \trianglelefteq_{m-1}\}$ which realizes $(P \setminus C, \preceq)$. 
We construct a set of $m+2n$ linear orders which realizes $(P,\preceq)$.

For each $i < m$ we extend $\trianglelefteq_i$ to a linearization $\trianglelefteq^*_i$ of $(P, \preceq)$ as we did in Theorem \ref{hira-cn}.

We then apply Lemma \ref{chains} to the poset $P$ and to the $2n$ pairs of chains $C_i, \emptyset$ and $\emptyset, C_i$ for $i < n$.
We obtain linearizations $\{\trianglelefteq^*_{m+j}\}_{j < 2n}$ such that if $j < 2n$ is even (respectively, odd) then $\trianglelefteq^*_{m+j}$ is the linearization of $P$ that puts $C_{\frac{i}{2}}$ at the bottom (respectively, that puts $C_{\frac{i-1}{2}}$ at the top).

Showing that the set $\{\trianglelefteq^*_0, \ldots, \trianglelefteq^*_{m+2n-1}\}$ realizes $(P,\preceq)$ is similar to the analogous proof in Theorem \ref{hira-cn}.
\end{proof}

Now we deal with the reversals starting with $\DBio$ and $\DBit$.

\begin{theorem}[$\RCA$]\label{reversaldbcn}
The following are equivalent:
\begin{enumerate}[leftmargin=*,label=(\arabic*)]
    \item \label{4.2.1} $\WKL$;
    \item \label{4.2.2n} $\forall n\, \DBin$;
    \item \label{4.2.2} $\DBit$;
    \item \label{4.2.3n} $\forall n\, \DBcn$;
    \item \label{4.2.3} $\DBio$.
\end{enumerate}
\end{theorem}

\begin{proof}
Theorem \ref{hira-cn} proves that \ref{4.2.1} implies \ref{4.2.2n}, while Theorem \ref{comparable chains} shows that \ref{4.2.1} implies \ref{4.2.3n}.
\ref{4.2.2n} implies \ref{4.2.2} and \ref{4.2.3n} implies \ref{4.2.3} are obvious (notice that $\DBio$ and $\DBc 1$ coincide).
The proof of Corollary \ref{hira-d} shows that \ref{4.2.2} implies \ref{4.2.3}.
Therefore we are left to prove that \ref{4.2.3} implies \ref{4.2.1}.

Let $f,g$ be 1-1 functions with disjoint ranges: we want to show that there exists a set $A$ such that $\ran(f) \subseteq A$ and $A \cap \ran(g) = \emptyset$. 
We consider the poset $P^4$, which is the case $k=4$ of the poset defined in the proof of Theorem \ref{th:dimension}.

If $m \notin \ran(f) \cup \ran(g)$ then $P^4_m = \{x^m, y^m\}$ is an antichain.
If $f(r)=m$ or $g(s)=m$, then the level $P^4_m$ is one of the two copies of $F_4$ below:

\adjustbox{scale=1,center}{
\begin{tikzcd}
{y^m} & {d^r_0} & {d^r_1} & {d^r_2} \\
	\\
	{x^m} & {c^r_0} & {c^r_1} & {c^r_2}
	\arrow[no head, from=3-1, to=1-2]
	\arrow[no head, from=3-1, to=1-3]
	\arrow[no head, from=3-1, to=1-4]
	\arrow[no head, from=3-2, to=1-1]
	\arrow[no head, from=3-2, to=1-3]
	\arrow[no head, from=3-2, to=1-4]
	\arrow[no head, from=3-3, to=1-1]
	\arrow[no head, from=3-3, to=1-2]
	\arrow[no head, from=3-3, to=1-4]
	\arrow[no head, from=3-4, to=1-3]
	\arrow[no head, from=3-4, to=1-2]
	\arrow[no head, from=3-4, to=1-1]
\end{tikzcd}}

\adjustbox{scale=1,center}{
\begin{tikzcd}
{x^m} & {q^s_0} & {q^s_1} & {q^s_2} \\
	\\
	{y^m} & {p^s_0} & {p^s_1} & {p^s_2}
	\arrow[no head, from=3-1, to=1-2]
	\arrow[no head, from=3-1, to=1-3]
	\arrow[no head, from=3-1, to=1-4]
	\arrow[no head, from=3-2, to=1-1]
	\arrow[no head, from=3-2, to=1-3]
	\arrow[no head, from=3-2, to=1-4]
	\arrow[no head, from=3-3, to=1-1]
	\arrow[no head, from=3-3, to=1-2]
	\arrow[no head, from=3-3, to=1-4]
	\arrow[no head, from=3-4, to=1-3]
	\arrow[no head, from=3-4, to=1-2]
	\arrow[no head, from=3-4, to=1-1]
\end{tikzcd}}

Albeit we can realize each level with at most $4$ linearizations, in $\RCA$ we cannot define a set of four linearizations of $\preceq$ that realize $\dim(P^4, \preceq)$.
Indeed, to do so we need to know whether $m \in \ran(f)$ (in which case Lemma \ref{lemmino} implies that three of the linearizations need to put $x^m$ below $y^m$) or $m \in \ran(g)$ (in which case three of the linearizations need to put $x^m$ above $y^m$).

Let $C = \{c^r_1,d^r_2,p^s_1,q^s_2: r,s \in \N\}$ and notice that it is a chain.
We now consider the poset $(P^4 \setminus C, \preceq)$.
If $m \in \ran(f)$ or $m \in \ran(g)$ then $P^4_m \setminus C$ has one of the following forms.
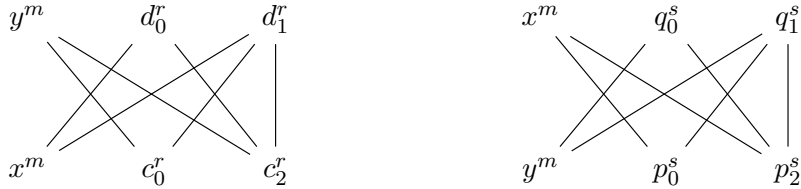
\[\begin{tikzcd}
	{y^m} & {d_0^r} & {d_1^r} &&& {x^m} & {q^s_0} & {q_1^s} \\
	\\
	{x^m} & {c_0^r} & {c_2^r} &&& {y^m} & {p_0^s} & {p_2^s}
	\arrow[no head, from=3-1, to=1-2]
	\arrow[no head, from=3-1, to=1-3]
	\arrow[no head, from=3-2, to=1-1]
	\arrow[no head, from=3-2, to=1-3]
	\arrow[no head, from=3-3, to=1-3]
	\arrow[no head, from=3-3, to=1-2]
	\arrow[no head, from=3-3, to=1-1]
	\arrow[no head, from=3-6, to=1-7]
	\arrow[no head, from=3-6, to=1-8]
	\arrow[no head, from=3-7, to=1-6]
	\arrow[no head, from=3-7, to=1-8]
	\arrow[no head, from=3-8, to=1-8]
	\arrow[no head, from=3-8, to=1-7]
	\arrow[no head, from=3-8, to=1-6]
\end{tikzcd}\]

We claim that $\dim(P^4 \setminus C, \preceq) = 2$.
First notice that for each $m$ we have that $x^m,y^m \notin C$ and $x^m \mid y^m$.
Therefore $\dim(P^4 \setminus C,\preceq) \ge 2$.
We now exhibit two linearizations $\trianglelefteq_0$ and $\trianglelefteq_1$ that realize $(P^4 \setminus C, \preceq)$.
We define $\trianglelefteq_0$ by:
\begin{itemize}    
    \item if $u \preceq v$ then $u \trianglelefteq_0 v$;
    \item $x^m \trianglelefteq_0 y^m$;
    \item for each $r$ if $f(r) = m$ then $x^m \trianglelefteq_0 c^r_2 \trianglelefteq_0 d^r_0 \trianglelefteq_0 c^r_0 \trianglelefteq_0 d^r_1 \trianglelefteq_0 y^m$;
    \item for each $s$ if $g(s) = m$ then $p^s_0 \trianglelefteq_0 p^s_2 \trianglelefteq_0 x^m \trianglelefteq_0 y^m \trianglelefteq_0 q^s_1 \trianglelefteq_0 q^s_0$.
\end{itemize}
In a similar fashion, we define $\trianglelefteq_1$:
\begin{itemize}
    \item if $u \preceq v$ then $u \trianglelefteq_1 v$;
    \item $y^m \trianglelefteq_1 x^m$;
    \item for each $r$ if $f(r) = m$ then $c^r_0 \trianglelefteq_1 c^r_2 \trianglelefteq_1 y^m \trianglelefteq_1 x^m \trianglelefteq_1 d^r_1 \trianglelefteq_1 d^r_0$;
    \item for each $s$ if $g(s) = m$ then $y^m \trianglelefteq_1 p^s_2 \trianglelefteq_1 q^s_0 \trianglelefteq_1 p^s_0 \trianglelefteq_1 q^s_1 \trianglelefteq_1 x^m$.
\end{itemize}
To prove that $\trianglelefteq_0$ and $\trianglelefteq_1$ realize the poset, we need to show that if $u \mid v$ then either $u \ntrianglelefteq_0 v$ or $u \ntrianglelefteq_1 v$. 
Since $u \mid v$ implies $\ell(u) = \ell(v)$ there are at most $16$ cases, that are easily checked.

Now, \ref{4.2.3} implies that $\dim(P^4, \preceq) \le 4$ and,  as in the proof of Theorem \ref{th:dimension}, a set separating the ranges of $f$ and $g$ is $\{m : |\{i < 4 : y^m \trianglelefteq^*_i x^m\}| =1 \}$.
\end{proof}

The key point of the construction in Theorem \ref{reversaldbcn} is that $\dim (P \setminus C, {\preceq}) = 2$.
Therefore, when we add back the chain $C$ and apply $\DBi1$, we obtain a set of four linearizations which realizes $P^4$ and we can exploit Lemma \ref{lemmino} to define the separator set $A$.
The same construction can be used to prove that $\DBith$ is equivalent to $\WKL$ too.

\begin{corollary}[$\RCA$]\label{reversaldbc3}
$\WKL$ is equivalent to $\DBi3$.
\end{corollary}

\begin{proof}
The forward direction is proved in Theorem \ref{hira-cn}.

For the backward direction, given $f$ and $g$ we consider the same poset $(P^4, {\preceq})$ of the proof of Theorem \ref{reversaldbcn}.
Let $C_0$ be the chain called $C$ in that proof and let $C_1 = C_2 = \emptyset$ (viewed as a chain). 
The poset $(P^4 \setminus (C_0 \cup C_1 \cup C_2), {\preceq})$ coincides with $(P^4 \setminus C, \preceq)$ and has dimension $2$.
By $\DBith$ we get that $\dim(P^4,{\preceq}) \le 5$ and we can fix a set $\{\trianglelefteq_0, \trianglelefteq_1, \trianglelefteq_2, \trianglelefteq_3, \trianglelefteq_4\}$ of linearizations which realizes $(P^4, {\preceq})$.

If $f(r) = m$ for some $r$, then three distinct linearizations have to deal with the incomparabilities $c^r_0 \mid d^r_0$, $c^r_1 \mid d^r_1$ and $c^r_2 \mid d^r_2$ in $P^4_m$.
This means that if $m \in \ran(f)$ then $|\{i < 5 : y^m \trianglelefteq^*_i x^m\} | \le 2$.
Analogously if $m \in \ran(g)$ then $| \{i < 5 : y^m \trianglelefteq^*_i x^m\} | \ge 3$.
Therefore we can define the separator set $A$ as the set of $m \in \N$ such that $| \{i < 5 : y^m \trianglelefteq^*_i x^m\} | \le 2$.  
\end{proof}

The poset of Theorem \ref{th:dimension} and idea of Theorem \ref{reversaldbcn} can be further exploited to obtain a reversal for every $\DBcn$.

\begin{theorem}\label{reversalwbcn}
For each $n$, $\RCA$ proves the equivalence between $\DBcn$ and $\WKL$.
\end{theorem}

\begin{proof}
Theorem \ref{comparable chains} shows one implication, so we deal with the converse. 

If $n = 1$ then $\DBc1$ is $\DBio$ and the implication was proved in Theorem \ref{reversaldbcn}.
So fix $n > 1$ and as usual let $f,g$ be 1-1 functions with disjoint ranges and consider the poset $P^{2n+2}$, which is the case $k=2n+2$ of the poset defined in the proof of Theorem \ref{th:dimension}.

If $m \notin \ran(f) \cup \ran(g)$ then $P^{2n+2}_m = \{x^m, y^m\}$ is an antichain.
If $f(r)=m$ or $g(s)=m$, then the level $P^{2n+2}_m$ is one of the two copies of $F_{2n+2}$ below:
\[\begin{tikzcd}
	{y^m} & {d_0^r} & \cdots & {d_{2n}^r} \\
	\\
	{x^m} & {c_0^r} & \cdots & {c_{2n}^r}
	\arrow[no head, from=3-1, to=1-2]
	\arrow[no head, from=3-1, to=1-4]
	\arrow[no head, from=3-2, to=1-1]
	\arrow[no head, from=3-2, to=1-4]
	\arrow[no head, from=3-4, to=1-2]
	\arrow[no head, from=3-4, to=1-1]
\end{tikzcd}\]

\[\begin{tikzcd}
	{x^m} & {q_0^s} & \cdots & {q_{2n}^s} \\
	\\
	{y^m} & {p_0^s} & \cdots & {p_{2n}^s}
	\arrow[no head, from=3-1, to=1-2]
	\arrow[no head, from=3-1, to=1-4]
	\arrow[no head, from=3-2, to=1-1]
	\arrow[no head, from=3-2, to=1-4]
	\arrow[no head, from=3-4, to=1-2]
	\arrow[no head, from=3-4, to=1-1]
\end{tikzcd}\]

For $1 \le i \le n$ consider the chain $C_i = \{c^r_{2i}, d^r_{2i-1}, p^s_{2i}, q^s_{2i-1}: r,s \in \N\}$.
Let $C = \bigcup_{i = 1}^n C_i$.
If $m \in \ran(f)$ or $m \in \ran(g)$ then $P^{2n+2}_m \setminus C$ has one of the following forms.
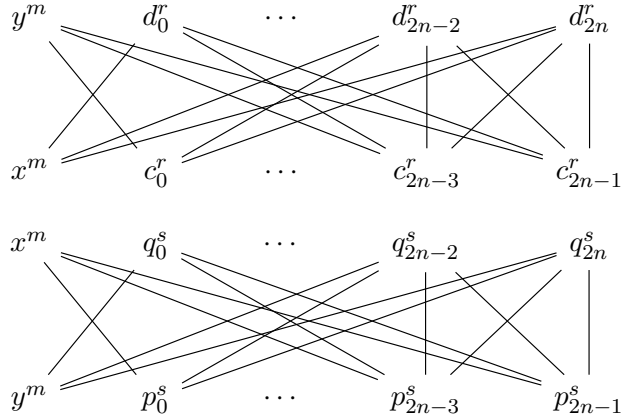
\[\begin{tikzcd}
	{y^m} & {d_0^r} & \cdots & {d_{2n-2}^r} & {d_{2n}^r} \\
	\\
	{x^m} & {c_0^r} & \cdots & {c_{2n-3}^r} & {c^r_{2n-1}}
	\arrow[no head, from=3-1, to=1-2]
	\arrow[no head, from=3-1, to=1-5]
	\arrow[no head, from=3-2, to=1-1]
	\arrow[no head, from=3-2, to=1-5]
	\arrow[no head, from=3-4, to=1-1]
	\arrow[no head, from=3-4, to=1-2]
	\arrow[no head, from=3-4, to=1-5]
	\arrow[no head, from=3-1, to=1-4]
	\arrow[no head, from=3-2, to=1-4]
	\arrow[no head, from=3-4, to=1-4]
	\arrow[no head, from=3-5, to=1-5]
	\arrow[no head, from=3-5, to=1-4]
	\arrow[no head, from=3-5, to=1-2]
	\arrow[no head, from=3-5, to=1-1]
\end{tikzcd}\]

\[\begin{tikzcd}
	{x^m} & {q_0^s} & \cdots & {q_{2n-2}^s} & {q_{2n}^s} \\
	\\
	{y^m} & {p_0^s} & \cdots & {p_{2n-3}^s} & {p^s_{2n-1}}
	\arrow[no head, from=3-1, to=1-2]
	\arrow[no head, from=3-1, to=1-5]
	\arrow[no head, from=3-2, to=1-1]
	\arrow[no head, from=3-2, to=1-5]
	\arrow[no head, from=3-4, to=1-1]
	\arrow[no head, from=3-4, to=1-2]
	\arrow[no head, from=3-4, to=1-5]
	\arrow[no head, from=3-1, to=1-4]
	\arrow[no head, from=3-2, to=1-4]
	\arrow[no head, from=3-4, to=1-4]
	\arrow[no head, from=3-5, to=1-5]
	\arrow[no head, from=3-5, to=1-4]
	\arrow[no head, from=3-5, to=1-2]
	\arrow[no head, from=3-5, to=1-1]
\end{tikzcd}\]

We claim that $\dim(P^{2n+2} \setminus C, \preceq) = 2$.
First notice that $P^{2n+2} \setminus C$ is not a chain and so $\dim(P^{2n+2}  \setminus C, \preceq) \ge 2$.
Hence it suffices to exhibit two linearizations $\trianglelefteq_0$ and $\trianglelefteq_1$ that realize $(P^{2n+2}  \setminus C, \preceq)$.
We define $\trianglelefteq_0$ in the following way:
\begin{itemize}
    \item if $z_1 \preceq z_2$ then $z_1 \trianglelefteq_0 z_2$;
    \item $x^m \trianglelefteq_0 y^m$;
    \item for each $r$ if $f(r) = m$ then
    $$x^m \trianglelefteq_0 c^r_{2n-1} \trianglelefteq_0 \ldots \trianglelefteq_0 c^r_1 \trianglelefteq_0 d^r_0 \trianglelefteq_0 c^r_0 \trianglelefteq_0 d^r_2 \trianglelefteq_0 \ldots \trianglelefteq_0 d^r_{2n} \trianglelefteq_0 y^m;$$
    \item for each $s$ if $g(s) = m$ then 
    $$p^s_0 \trianglelefteq_0 \ldots \trianglelefteq_0 p^s_{2n-1} \trianglelefteq_0 x^m \trianglelefteq_0 y^m \trianglelefteq_0 q^s_{2n} \trianglelefteq_0 \ldots \trianglelefteq_0 q^s_0.$$
\end{itemize}
In a similar fashion, we define $\trianglelefteq_1$:
\begin{itemize}
    \item if $z_1 \preceq z_2$ then $z_1 \trianglelefteq_1 z_2$;
    \item $y^m \trianglelefteq_1 x^m$;
    \item for each $r$ if $f(r) = m$ then 
    $$c^r_0 \trianglelefteq_1 \ldots \trianglelefteq_1 c^r_{2n-1} \trianglelefteq_1 y^m \trianglelefteq_1 x^m \trianglelefteq_1 d^r_{2n} \trianglelefteq_1 \ldots \trianglelefteq_1 d^r_0;$$
    \item for each $s$ if $g(s) = m$ then 
    $$y^m \trianglelefteq_0 p^s_{2n-1} \trianglelefteq_0 \ldots \trianglelefteq_0 p^s_1 \trianglelefteq_0 q^s_0 \trianglelefteq_0 p^s_0 \trianglelefteq_0 q^s_2 \trianglelefteq_0 \ldots \trianglelefteq_0 q^s_{2n} \trianglelefteq_0 y^m.$$
\end{itemize}
It is easy to verify that $\trianglelefteq_0$ and $\trianglelefteq_1$ realize $(P^{2n+2}  \setminus C, \preceq)$.

By $\DBcn$ we know that $\dim(P^{2n+2}, \preceq) \le 2n + 2$ and the set separating the ranges of $f$ and $g$ is defined as $A = \{m : |\{i < 2n+2 : y^m \trianglelefteq^*_i x^m\}| =1 \}$ as in the proof of Theorem \ref{th:dimension}.
\end{proof}

We still need to show that $\DBin$ for $n \ge 4$ implies $\WKL$. 
The poset $P^k$ of Theorem \ref{th:dimension} and the idea of the proof of Theorem \ref{reversalwbcn} cannot be used because in that poset if a chain intersects two or more levels then it is comparable with every other nonempty chain.
We need a crucial modification.

\begin{theorem}
For each $n$, $\RCA$ proves the equivalence between $\DBin$ and $\WKL$.
\end{theorem}

\begin{proof}
Theorem \ref{hira-cn} shows one implication, so we deal with the converse.

We already proved the reversal for $n<4$ in Theorem \ref{reversaldbcn} and Corollary \ref{reversaldbc3}.
The proof we are about to give works for $n \ge 3$.
As usual let $f$ and $g$ be 1-1 functions with disjoint ranges.
Let
$$P=\{x^i,y^i: i \in \N\} \cup \{c^r_j, d^r_j : j < n, r \in \N\} \cup \{p^s_j, q^s_j : j < n, s \in \N\}$$
and define $\ell: P \rightarrow \N$ as in Theorem \ref{th:dimension}.
As before, we let  $P_m$ be the level of elements $z$ such that $\ell(z) = m$.

For each $z_1, z_2 \in P$ we stipulate the following:
\begin{itemize}
	\item If neither $z_1$ nor $z_2$ is a $d^r_j$ or a $q^s_k$, then $z_1 \prec z_2$ if and only if either $\ell (z_1) < \ell (z_2)$ or $z_1 = c^r_j$ and $z_2 = y^{f(r)}$ or $z_1 = p^r_j$ and $z_2 = x^{f(r)}$ (this case coincides with what we did in the previous proofs).
	\item If $z_1 = d^r_j$ then $z_1 \prec z_2$ if and only if $z_2=d^k_j$ for some $k$ such that $f(r) < f(k)$ or $z_2=q^s_j$ for some $s$ such that $f(r) < g(s)$ (namely, $z_2$ is a $d^k_j$ or a $q^s_j$ for the same $j$ of $z_1$ and belongs to some higher level).
	If $z_1 = q^s_j$ the definition is analogous.
	\item If $z_2 = d^r_j$ then $z_1 \prec z_2$ if and only if $z_1$ is either $x^{f(r)}$ or $c^r_i$ for $i \ne j$ or $\ell(z_1) < f(r)$ and if $z_1$ is a $d^v_i$ or a $q^s_i$ then $i = j$.
	If $z_2=q^s_j$ the definition is analogous, replacing $x^m$ with $y^m$ and $c^r_i$ with $p^s_i$.
\end{itemize}
Each level $P_m$ is either an antichain of two elements or a copy of $F_{n+1}$.
The difference with the previous constructions lies in the comparabilities between different levels as shown in the figure below, where we are assuming that $f(r) = l$ and $g(s) = m$.
We highlight with a dashed line some of the comparabilities between $d^r_i$'s and $q^s_j$'s: notice that $d^r_0 \nprec p^s_0$ even though $\ell(d^r_0) < \ell(p^s_0)$.

\smallskip
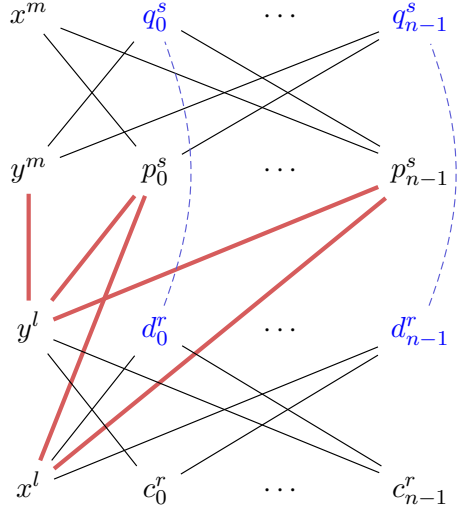
\adjustbox{scale=1,center}{
\begin{tikzcd}
	{x^m} & {\textcolor{blue}{q^s_0}} & \cdots & {\textcolor{blue}{q^s_{n-1}}} \\
	\\
	{y^m} & {p^s_0} & \cdots & {p^s_{n-1}} \\
	\\
	{y^l} & {\textcolor{blue}{d^r_0}} & \cdots & {\textcolor{blue}{d^r_{n-1}}} \\
	\\
	{x^l} & {c^r_0} & \cdots & {c^r_{n-1}}
	\arrow[no head, from=3-1, to=1-2]
	\arrow[no head, from=3-1, to=1-4]
	\arrow[no head, from=3-2, to=1-1]
	\arrow[no head, from=3-2, to=1-4]
	\arrow[no head, from=3-4, to=1-1]
	\arrow[no head, from=3-4, to=1-2]
	\arrow[color={rgb,255:red,214;green,92;blue,92}, ultra thick, no head, from=5-1, to=3-1]
	\arrow[color={rgb,255:red,214;green,92;blue,92}, ultra thick, no head, from=5-1, to=3-2]
	\arrow[color={rgb,255:red,214;green,92;blue,92}, ultra thick, no head, from=5-1, to=3-4]
	\arrow[color={rgb,255:red,92;green,92;blue,214}, dashed, no head, from=5-2, to=1-2, bend right=20]
	\arrow[color={rgb,255:red,92;green,92;blue,214}, dashed, no head, from=5-4, to=1-4, bend right=20]
	\arrow[color={rgb,255:red,214;green,92;blue,92}, ultra thick, no head, from=7-1, to=3-2]
	\arrow[color={rgb,255:red,214;green,92;blue,92}, ultra thick, no head, from=7-1, to=3-4]
	\arrow[no head, from=7-1, to=5-2]
	\arrow[no head, from=7-1, to=5-4]
	\arrow[no head, from=7-2, to=5-1]
	\arrow[no head, from=7-2, to=5-4]
	\arrow[no head, from=7-4, to=5-1]
	\arrow[no head, from=7-4, to=5-2]
\end{tikzcd}}
\smallskip

The verification that we defined a poset is tedious but straightforward.


For each $j < n$ consider the chain $C_j=\{d^r_j, q^s_j: r,s \in \N\}$.
By definition of $\preceq$ these chains are incomparable.
Let $C= \bigcup_{j<n} C_j$ and consider the poset $(P \setminus C, \preceq)$.
If $m \in \ran(f)$ or $m \in \ran(g)$ then $P_m \setminus C$ has one of the following forms.
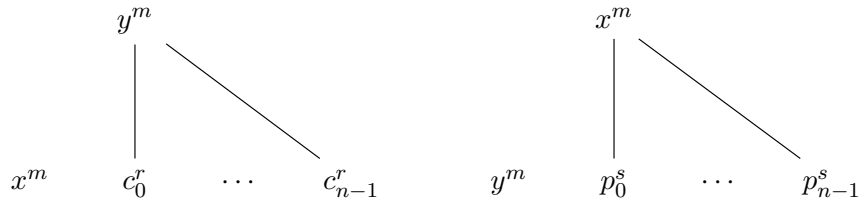
\[
\begin{tikzcd}[column sep=1.5em, row sep=2em]
	& {y^m} &&&&& {x^m} \\
	\\
	{x^m} & {c^r_0} & \cdots & {c^r_{n-1}} & & {y^m} & {p^s_0} & \cdots & {p^s_{n-1}}
	\arrow[no head, from=1-2, to=3-2] 
	\arrow[no head, from=1-2, to=3-4] 
	\arrow[no head, from=1-7, to=3-7] 
	\arrow[no head, from=1-7, to=3-9] 
\end{tikzcd}
\]
It is easy to check that  $\dim(P \setminus C, {\preceq}) = 2$.
By $\DBin$ we have that $\dim(P, {\preceq}) \le n+2$ and we can fix a set $\{\trianglelefteq^*_0, \ldots, \trianglelefteq^*_{n+1}\}$ of $n+2$ linearizations which realizes $(P, {\preceq})$. 
If $m \in \ran(f)$, then it must be $| \{i < n+2 : y^m\trianglelefteq^*_i x^m\} | \le 2$ and analogously, if $m \in \ran(g)$, then it must be $| \{i < n+2 : y^m \trianglelefteq^*_i x^m\} | \ge n$.
Since we are supposing $n \ge 3$, the set of $m \in \N$ such that $| \{i < n : y^m \trianglelefteq^*_i x^m\} | \le 2$ separates $\ran(f)$ and $\ran(g)$.
\end{proof}

\section{$\DBp$}\label{sec:point}

We now deal with $\DBp$ which was introduced in Section \ref{sec:boundth}.
We recall it: for each poset $(P, \preceq)$ and each $x_0 \in P$, $\dim(P,\preceq) \le \dim(P \setminus \{x_0\},\preceq)+1$.

Let $(P, \preceq)$ be a poset and $I,F \subseteq P$.
We write $I \prec F$ if for each $i \in I$ and each $f \in F$ it holds that $i \prec f$.
An initial interval of $(P, \preceq)$ is a $\preceq$-downward closed set.
An initial interval $B \subseteq P$ that contains $I$ and is disjoint from $F$ is called a separator set for $P,I,F$.
An element $b \in P$ such that $\forall i \in I \, \forall f \in F \, (i \preceq b \preceq f)$ is called a separator element for $P,I,F$.
For simplicity we just say that $B$ is a separator set and $b$ is a separator element without specifying the order and the sets considered when these are clear from the context.

Being able to separate subsets in a poset or in a linear order is crucial for the rest of the section.
We first recall a very general result that was proved in \cite{frittaionmarcone}.

\begin{lemma}[$\RCA$]\label{fm14}
The following are equivalent.
\begin{enumerate}[leftmargin=*,label=(\arabic*)]
    \item $\WKL$.
    \item \label{fritmarc} For each poset $(P, \preceq)$ and sets $I,F \subseteq P$ such that  $f \nprec i$ for each $i \in I$, $f \in F$, there exists a separator set.
\end{enumerate}
\end{lemma}

We are interested in studying a weaker version of \ref{fritmarc} of Lemma \ref{fm14} that we call $\mathsf{LS}$.
\begin{itemize}
    \item $\mathsf{LS}$ : for each linear order $(L, \trianglelefteq)$ and each sets $I,F \subseteq L$ such that $I \lhd F$, there exists a separator set $B$ for $L,I,F$.
\end{itemize}

\begin{lemma}\label{separatorrca}
$\RCA \vdash \mathsf{LS}$.
\end{lemma}

\begin{proof}
Suppose first that there exists a separator element $b \in L$.
If $b \in F$ then let $B = \{x \in L : x \lhd b\}$, otherwise let $B = \{x \in L : x \trianglelefteq b\}$.
Then $B$ exists in $\RCA$ and has the desired properties.

If there are no separator elements, let $B = \{x \in L : \exists i \in I (x \trianglelefteq i) \}$.
By hypothesis $x \in B$ if and only if $\forall f \in F (x \lhd f)$.
Therefore, $B$ can be defined by $\Delta^0_1$ comprehension and again has the required properties.
\end{proof}

We also need the finite parallelization of $\mathsf{LS}$ which we denote by $\mathsf{LS}^*$ following notational conventions from Weihrauch reducibility.
\begin{itemize}
    \item $\mathsf{LS}^*$ : for each $k$ and each finite sequence $(L_j, I_j, F_j)_{j < k}$ where each $(L_j, \trianglelefteq_j)$ is a linear order and $I_j \lhd_j F_j$, there exists a sequence $(B_j)_{j < k}$ such that each $B_j$ is a separator set for $L_j, I_j, F_j$.
\end{itemize}

Notice that the proof of Lemma \ref{separatorrca} relies on non uniform information: whether or not a separator element $b \in P$ exists.
Therefore $\mathsf{LS}^*$ may not be available in $\RCA$ and we will show in Corollary \ref{reversemathresult} that this is the case.

As we will see, $\mathsf{LS}^*$ is tightly connected to $\DBp$.
We show that $\mathsf{LS}^*$ is equivalent to a strong version of $\DBp$ that we call $\DBpp$.
\begin{itemize}
    \item $\DBpp$: for each poset $P$ and each $x_0 \in P$, if $(L_j)_{j < n}$ realizes $P \setminus \{x_0\}$ then there exist $(L_j')_{j < n+1}$ which realize $P$ and such that for $j < n-1$, $L_j = L'_j \setminus \{x_0\}$. 
\end{itemize}

$\DBpp$ establishes a connection between the realization of $P \setminus \{x_0\}$ and the realization of $P$ we obtain.
$\DBp$ does not require this very natural correlation, which actually occurs in every known proof of $\DBp$.
Hence, $\DBpp$ is a reasonable strengthening of $\DBp$.

\begin{theorem}[$\RCA$]\label{sep-dbp}
$\mathsf{LS}^*$ is equivalent to $\DBpp$.
\end{theorem}

\begin{proof}
For the forward direction, let $(P, \preceq)$ be a poset.
Let $x_0 \in P$ and let $\{\trianglelefteq_0, \ldots, \trianglelefteq_{n-1}\}$ be a realization of $(P \setminus \{x_0\}, \preceq)$.

Let $I=\{x: x \in P \wedge x \prec x_0\}$ and $F= \{x: x \in P \wedge x_0 \prec x \}$ which exist in $\RCA$ and are respectively downward and upward closed in $P$.
Notice that by transitivity $I \prec F$ and since ${\prec} \subseteq {\lhd_i}$ we have $I \lhd_i F$ for each $i < n$.
Define two linearizations $\trianglelefteq_{n-1}^0,\trianglelefteq_{n-1}^1$ of $(P, \preceq)$ starting from $\trianglelefteq_{n-1}$ by
$$I \trianglelefteq_{n-1}^0 \{x_0\} \trianglelefteq_{n-1}^0 P \setminus (\{x_0\} \cup I)$$
$$P \setminus (\{x_0\} \cup F) \trianglelefteq_{n-1}^1 \{x_0\} \trianglelefteq_{n-1}^1 F$$
where each segment is ordered by $\trianglelefteq_{n-1}$. 
Notice that $\RCA$ proves the existence of each segment and (using $I, F$ and $\trianglelefteq_{n-1}$ as parameters) of the linear orders $\trianglelefteq_{n-1}^0$ and $\trianglelefteq_{n-1}^1$.
It is clear that $\trianglelefteq_{n-1}^0$ and $\trianglelefteq_{n-1}^1$ are linearizations of $\preceq$.
Notice that if $y \trianglelefteq_{n-1} z$ then at least one of $y \trianglelefteq^0_{n-1} z$ and $y \trianglelefteq^1_{n-1} z$ holds.

Next we want to extend each $\trianglelefteq_i$ for $i < n-1$ to a linearization of $(P, \preceq)$.
To this end, we apply $\mathsf{LS}^*$ to the sequence $(\trianglelefteq_i, I, F)_{i < n-1}$ to obtain, for each $i < n-1$, a separator set $B_i$ with respect to $\trianglelefteq_i$.
Then we define linearizations $\trianglelefteq_i^*$ of $(P, \preceq)$ each extending the corresponding $\trianglelefteq_i$ by
$$B_i \trianglelefteq_i^* \{x_0\} \trianglelefteq_i^* P \setminus (\{x_0\} \cup B_i)$$
where each segment is ordered by $\trianglelefteq_i$.

Finally we claim that the set  $\{\trianglelefteq_0^*, \ldots, \trianglelefteq_{n-1}^0, \trianglelefteq_{n-1}^1\}$ realizes $(P, \preceq)$.
We need to prove that if $z \npreceq y$ then either $y \trianglelefteq^0_{n-1} z$ or $y \trianglelefteq^1_{n-1} z$ or $y \trianglelefteq_i^* z$ for some $i < n-1$.
If $y = x_0$ then $z \notin I$ and $y \trianglelefteq^0_{n-1} z$, while if $z = x_0$ then $y \notin F$ and $y \trianglelefteq^1_{n-1} z$.
If neither $y$ nor $z$ are $x_0$ we have $y \trianglelefteq_i z$ for some $i<n$.
If $i=n-1$ we already noticed that at least one of $y \trianglelefteq^0_{n-1} z$ and $y \trianglelefteq^1_{n-1} z$ holds.
If $i<n-1$ then, since $z \in B_i$ and $y \notin B_i$ cannot hold, we have $y \trianglelefteq_i^* z$.

\medskip

For the converse direction let $(L_j, I_j, F_j)_{j < n}$ be a sequence of linear orders with sets to be separated.
First we extend each $L_j$ in $\RCA$ so that $I_j$, $F_j$ and $L_j \setminus (I_j \cup F_j)$ are infinite.
This can be achieved by adding infinitely many new elements below $L_j$, stipulate that they belong to $I_j$, and adding infinitely many new elements above $L_j$ and stipulate that infinitely and coinfinitely many of them belong to $F_j$.
We still denote the new sequence by $(L_j, I_j, F_j)_{j < n}$ and notice that it suffices to find a solution for it.

Next we identify all the elements of $I_j$ (respectively, $F_j$) throughout all the linear orders $L_j$ as follows.
We may assume that the support of each $L_j$ is $\N$.
Let $f_j \colon \N \to \N \setminus \{0\}$ be the bijection defined as
$$f_j(y) = 
\begin{cases}
\text{the least unused } n \equiv 0 \mod 3 & \text{if } y \in I_j; \\
\text{the least unused } n \equiv 1 \mod 3 & \text{if } y \in F_j; \\
\text{the least unused } n \equiv 2 \mod 3 & \text{otherwise.}
\end{cases}$$
We define the linear order $\lhd_j$ on $\N \setminus \{0\}$ as $n \lhd_j m$ if and only if $f_j^{-1} (n)$ is below $f_j^{-1} (m)$ in $L_j$.
Since $f_j$ is a computable isomorphism, the new linear orders (which we still denote by $L_j$) exists in $\RCA$.
Now for each $j$ we have $I_j = \{n > 0 : n \equiv 0 \mod 3\}$ and $F_j = \{n : n \equiv 1 \mod 3\}$.
We simply call $I$ and $F$ those sets.
Moreover, it suffices to find a solution for the sequence $(L_j, I, F)_{j < n}$.

We define an additional linear order $L_n$ on $\N \setminus \{0\}$: we stipulate $y \lhd_n z$ if and only if either $y \equiv z \mod 3$ and $y < z$ (in the standard ordering of $\N$), or $y \equiv 0 \mod 3$ and $z$ is not, or $y \equiv 2 \mod 3$ and $z \equiv 1 \mod 3$.
In other words in $L_n$ the elements of $I$ are followed by the elements of $\N \setminus (I \cup F)$, which in turn are followed by the elements of $F$.

Let $\prec$ be the binary relation on $\N$ obtained by taking the intersection of the $L_j$'s for $j < n+1$, and then adding $0$ so that $I \prec 0 \prec F$ and $0$ is incomparable with everything else.
We claim that $P = (\N, \prec)$ is a poset.
The only non trivial facts to prove is that if $z \prec y$ and $y \prec 0$ then $z \prec 0$, and its dual if $0 \prec y$ and $y \prec z$ then $0 \prec z$.
In the first case, by definition $y \in I$ and by construction of $L_n$ it must be $z \in I$ too.
So $z \prec 0$ by definition.
The dual is proved in the same way, using $F$ in place of $I$.

By definition of $P$ the sequence $(L_j)_{j<n+1}$ realizes $P \setminus \{0\}$.
By the statement, let $(L'_j, \lhd_j')_{j < n+2}$ be a realization of $P$ such that for $j < n$ $\lhd_j = \lhd_j' \restriction \N \setminus \{0\}$.
Then we uniformly separate $(L_j, I_j, F_j)_{j < n}$ by setting $B'_j = \{x \in L'_j : x \lhd_j' 0\} \subset L_j$.
\end{proof}

We immediately derive the following.

\begin{corollary}[$\RCA$]
$\mathsf{LS}^* \vdash \DBp$.
\end{corollary}

We now focus on studying $\mathsf{LS}^*$.
We already know that $\WKL \vdash \mathsf{LS}^*$ by a simple application of Lemma \ref{fm14}.
The next result implies that the reverse implication cannot hold.

\begin{lemma}[$\RCA$]\label{leastsepel}
The following are equivalent.
\begin{enumerate}[leftmargin=*,label=(\arabic*)]
    \item \label{5.1.1} $\I \SI02$.
    \item \label{5.1.2} For each sequence of linear orders with sets to be separated $(L_i,I_i,F_i)_{i < n}$, there exists the set $X = \{j < n : \exists b \in L_j \, \forall i \in I_j \, \forall f \in F_j \, (i \trianglelefteq_j b \trianglelefteq_j f)\}$.
\end{enumerate}
\end{lemma}

\begin{proof}
\ref{5.1.1} implies \ref{5.1.2} because the existence of $X$ is an instance of bounded $\SI02$ comprehension which is equivalent to $\I \SI02$ (Theorem \ref{isigma}).

For the converse, we show that \ref{5.1.2} implies bounded $\SI02$ comprehension.
Let $\varphi(j)$ be a $\SI02$ formula of the form $\exists x \, \forall y \, \psi(j,x,y)$.
For each $n$, we need to prove that the set $\{j < n : \varphi(j)\}$ exists.
We aim to construct a finite sequence of $n$ linear orders $(L_j, \trianglelefteq_j)_{j < n}$, each with subsets $I_j, F_j$ to be separated.
We perform the construction to satisfy the following: for each $j < n$, $\varphi(j)$ holds if and only if the linear order $(L_j, \trianglelefteq_j)$ has a separator element.
Then it is clear that the set $X$ of \ref{5.1.2} is the set $\{j < n : \varphi(j) \}$.

Fix $j < n$.
The linear order $(L_j, \trianglelefteq_j)$ has domain $\N$, $I_j = \{m \in \N : m \equiv 0 \mod 3\}$ and $F_j = \{m \in \N : m \equiv 1 \mod 3\}$.
We stipulate that $\forall l \in I_j \, \forall m \in F_j \, (l \lhd_j m)$, $\forall l,m \in I_j \, (l \lhd_j m \leftrightarrow l < m)$ and $\forall l,m \in F_j \, (l \lhd_j m \leftrightarrow m < l)$ (where $<$ denotes the standard ordering of $\N$).
Since $I_j$ has no maximum and $F_j$ has no minimum, a separator element must belong to the set $Z_j = \{m \in \N : m \equiv 2 \mod 3\}$.
We also stipulate that $\forall l \in Z_j \, \forall m \in F_j \, (l \lhd_j m)$ and $\forall l,m \in Z_j \, (l \lhd_j m \leftrightarrow l < m)$

We design a strategy to define the comparabilities of the elements of $I_j$ and $Z_j$ in such a way to ensure that $\varphi(j)$ holds if and only if there is a separator element for $L_j, I_j, F_j$.
In other words, as long as $x$ appears to witness $\varphi(j)$, the strategy keeps $3x+2$ above all elements of $I_j$.
We proceed by stages: at stage $s$ we stipulate (at least) the comparabilities between each element of $Z_j$ and $3s \in I_j$.
The strategy keeps parameters $x_j^s$ for all $j < n$ to mark the current possible existential witness for $\varphi(j)$.
We start by setting $x_j^0 = 0$ for all $j < n$.\medskip

\textit{Stage $s$}. 
For each $j < n$ if $\forall y \le s \, \psi(j,x_j^s,y)$ holds, then we stipulate that $\forall m \ge x_j \, (3s \lhd_j 3m + 2)$ and let $x_j^{s+1} = x_j^s$.
On the other hand, if $\exists y \le s \, \neg \psi(j,x_j^s,y)$ holds, then we stipulate that $\forall t \ge s \, (3x_j^s + 2 \lhd_j 3t)$ and $\forall m > x_j^s (3s \lhd_j 3m + 2)$.
In this second case, we let $x_j^{s+1} = x_j^s+1$.

\medskip

This completes the construction of the linear orders $(L_j, \trianglelefteq_j)_{j < n}$.

We are left to prove that for $j < n$, $\varphi(j)$ holds if and only if the linear order $(L_j, \trianglelefteq_j)$ has a separator element.
For the forward direction, fix $j < n$ such that $\varphi(j)$ holds.
Let $\overline{x}$ be such that $\forall y \, \psi(j,\overline{x},y)$ holds.
By construction for each $s$ we stipulated that $3s  \lhd_j 3 \overline{x} +2$.
We conclude that $\overline{x}$ is a separator element for $L_j, I_j, F_j$.

Conversely, suppose that for $j < n$ $\neg \varphi(j)$ holds.
Since $\forall x \, \exists y \, \neg \psi(j,x,y)$ holds, by $\I \SI01$ we have that $\forall x \, \exists s \, (x = x^s_j)$.
Fix $z$ large enough so that $\exists y \le z \, \neg \psi(j,x,y)$ and $\exists s \le z \, (x = x_j^s)$.
Then $3 x + 2 \lhd_j 3 z$.
Since $x$ was arbitrary, there is no separator element for $L_j, I_j, F_j$.
\end{proof}

Lemma \ref{leastsepel} implies that assuming $\I \SI02$, Lemma \ref{separatorrca} can be applied uniformly an arbitrary finite number of times to produce a solution to $\mathsf{LS}^*$.
This is because, given a finite sequence of linear orders and of sets to be separated, $\I \SI02$ is able to recognize whether there is a separator element (which is the non computable information we need).
Once we know this, we can use the strategy of the proof of Lemma \ref{separatorrca} to produce a separator set for each linear order and to prove $\mathsf{LS}^*$.

Putting everything together, we obtain the following interesting result.

\begin{corollary}\label{dbpwkl}
$\WKL \vee \I \SI02 \vdash \mathsf{LS}^*$ and in particular $\WKL \vee \I \SI02 \vdash \DBp$.
\end{corollary}

Since $\WKL$ and $\I \SI02$ are incomparable, $\mathsf{LS}^*$ implies neither $\I \SI02$ nor $\WKL$.
Results equivalent to this disjunction do exist, but are rare in the literature.
The only examples known to the authors are in \cite{friedmansimpsonyu, belanger}.
In \cite{simpsonyokoyama} a weakening of weak König's lemma is shown to lie strictly between $\RCA$ and $\WKL \vee \I \SI02$.

Studying reversals for $\DBp$ is rather difficult because, as we already noticed, the statement does not presuppose a relation between the realization of $(P \setminus \{x_0\}, \preceq)$ and that of $(P, \preceq)$.
Moreover, it is not easy to code information in the position of the single point $x_0$.
However, in view of Theorem \ref{sep-dbp}, it is reasonable to study reversals for $\mathsf{LS}^*$ instead of $\DBp$.
Following ideas from \cite{simpsonyokoyama, patey2025}, we deal with a special case of $\mathsf{LS}^*$.
The notion of instance-solution problem is widely used in reverse mathematics.
These are typically sentences of the form
$$\forall X \, (\varphi(X) \rightarrow \exists Y \psi (X,Y))$$
where $X$ and $Y$ are second order variables which represent respectively the instances and the solution to the problem, and $\varphi$ and $\psi$ are arithmetical formulas.
$\mathsf{LS}^*$ can be formulated in this way.
An instance-solution problem $\mathsf{P}$ is computably true if every $\mathsf{P}$-instance $X$ admits a $\Delta_1(X)$-definable solution.
The proof of $\mathsf{LS}$ within $\I\SI{0}{2}$, shows that $\I\SI{0}{2}$ proves that $\mathsf{LS}^*$ is computably true.

\begin{theorem}\label{lowerbound}
Every model of $\RCA+\neg \B \SI02$ contains an instance $Z$ of $\mathsf{LS}^*$ with no $\Delta_1(Z)$-definable solution.  
\end{theorem}

\begin{proof}
Let $(M,S) \vDash \RCA+\neg \B \SI02$.
By the failure of $\B \SI02$ there exists $k \in M$ and $f \in S$ with $f \colon M \rightarrow k$ such that for each $i < k$, $f^{-1} (i)$ is finite.
Let $X \in S$ be such that $f$ is $\Delta_1(X)$ and let $(\Phi^X_e)_{e \in M}$ be an enumeration of the $X$-computable functions of the model (regarded as $\{0,1\}$-valued for convenience).
We write $\Phi^X_{e,s}$ to denote $\Phi^X_e$ after $s$ stages of the computation.

We construct by stages a $\Delta_1(X)$-instance $Z = (L_j, I_j, F_j)_{j < k}$ of $\mathsf{LS}^*$ such that for each $e \in M$, if $f(e) = j$ then $\Phi^X_e$ is not a separator set for $L_j, I_j, F_j$.
The support of each $L_j$ is $M$.
The construction keeps parameters $\delta^s_e$ for $e,s \in M$ and $x^s_j$ for $j < k$ and $s \in M$: here $s$ highlights that we are talking about the value of the parameter at the beginning of stage $s$.
Analogously, we denote by $L_j^s, I_j^s, F_j^s$ the finite approximations of $L_j, I_j, F_j$ at the beginning of stage $s$.
We start by setting $\delta^0_e = 0$, $x^0_j = 0$, $L_j^0 = \{x^0_j\}$ and $I_j^0 = F_j^0 = \emptyset$.
At each stage we add finitely many elements to $L_j$, define the comparabilities involving the new elements and decide which of them belong to $I_j$ or $F_j$.

For each $e \in M$ we want to satisfy the requirement $\mathsf{R}_e$ consisting of the disjunction of the following clauses where $f(e) = j$:
\begin{enumerate}[leftmargin=*,label=(\arabic*)]
    \item \label{p1} $\Phi^X_e(x,j) \uparrow$ for some $x \in M$;
    \item \label{p2} $\Phi^X_e(x,j) \downarrow = 0$ for some $x \in I_j$;
    \item \label{p3} $\Phi^X_e(x,j) \downarrow = 1$ for some $x \in F_j$;
    \item \label{p4} $\Phi^X_e(x,j) \downarrow = 1$ and $\Phi^X_e(y,j) \downarrow = 0$ for some $x,y \in L_j$ with $y \lhd_j x$ (here $\lhd_j$ denotes the order relation on $L_j$).
\end{enumerate}
Notice that if $\mathsf{R}_e$ is satisfied, then $\Phi^X_e$ is not the characteristic function of a separator set for $L_j,I_j,F_j$.
We set $\delta^s_e = 1$ if at stage $s$ we are sure that $\mathsf{R}_e$ is satisfied, no matter what happens at later stages. 

\smallskip

\textit{Stage $s$}. 
For every $j < k$ check if for some $e \in f^{-1}(j) \cap [0,s)$, $\Phi^X_{e,s} (x^s_j, j) \downarrow$ and $\delta^s_e = 0$.
Let $z^s_j$ and $w^s_j$ be the least elements of $M$ not in $L^s_j$.

If no such $e$ is found, set $x_j^{s+1} = x_j^s$, $L_j^{s+1} = L_j^s \cup \{z^s_j, w^s_j\}$, $I_j^{s+1} = I_j^s \cup \{z^s_j\}$, $F_j^{s+1} = F_j^s \cup \{w^s_j\}$ and put $z^s_j$ below every element of $L^s_j$ and $w^s_j$ above every element of $L^s_j$.

Otherwise let $e$ be least such that $\Phi^X_{e,s} (x^s_j, j) \downarrow$ and $\mathsf{R}_e$ is not yet satisfied.
The existence of a minimal $e$ is ensured by $\I \SI00$.
We say that $\mathsf{R}_e$ acts at this stage and we set $\delta^{s+1}_e=1$.

If $\Phi^X_{e,s} (x^s_j, j) \downarrow = 0$, put $x_j^s \lhd_j z^s_j \lhd_j w^s_j$.
We also stipulate that $z^s_j$ and $w^s_j$ are below all the elements of $L^s_j$ that were above $x^s_j$.
Then we set $x^{s+1}_j = w^s_j$, $L_j^{s+1} = L_j^s \cup \{z^s_j, w^s_j\}$, $I_j^{s+1} = I_j^s \cup \{z^s_j\}$, $F_j^{s+1} = F_j^s$.

If $\Phi^X_{e,s} (x^s_j, j) \downarrow = 1$, put $w_j^s \lhd_j z^s_j \lhd_j x^s_j$.
We also stipulate that $z^s_j$ and $w^s_j$ are above all the elements of $L^s_j$ that were below $x^s_j$.
Then we set $x^{s+1}_j = w^s_j$, $L_j^{s+1} = L_j^s \cup \{z^s_j, w^s_j\}$, $I_j^{s+1} = I_j^s$, $F_j^{s+1} = F_j^s \cup \{z^s_j\}$.

When we finish all the $j < k$, for each $e$ such that $\delta_e^{s+1}$ is still undefined, we set $\delta_e^{s+1} = \delta^s_e$ and we move to the next stage.
This completes the construction.

\smallskip

We need to prove that each requirement $\mathsf{R}_e$ is satisfied.
First notice that $\delta_e^{s+1} \neq \delta_e^s$ for at most one $s$ (the only stage when $\mathsf{R}_e$ act): when such $s$ exists we denote it by $s_e$.

We claim that for each $j < k$ there exists a stage $s$ and $x \in M$ such that for each stage $t > s$, $x^t_j = x^s_j$.
Suppose that this is not the case and let $j$ be such that for every stage $s$ and every $x$, there is a stage $t > s$ with $x^t_j \ne x^s_j$.
Fix a generic $s$.
Then by $\I \SI00$ there exists a least stage $t > s$ such that $x^s_j \ne x^t_j$ and we call $t_s$ this stage.
By minimality, $x^{t_s-1}_j = x^s_j$ and some requirement $\mathsf{R}_e$ for $e \in f^{-1}(j) \cap [0,t_s)$ has to act, otherwise $x^{t_s}_j = x^{t_s-1}_j$.
By construction, only one requirement acts and we call $e_s$ its index.
Moreover, each requirement acts at most once during the construction (exactly when $\delta_e^{s+1}$ becomes $1$) and so for each $s' < s$, $e_{s'} \ne e_s$.
It follows that the function $s \to e_s$ is an $X$ computable injection from $M$ to $f^{-1}(j)$ which contradicts the assumption that $f^{-1}(j)$ is finite and the claim is proved.

Now we show that each requirement $\mathsf{R}_e$ is satisfied.
If $\Phi^X_e$ does not converge on some input $(x, f(e))$, we meet clause \ref{p1}.
So suppose that $\Phi^X_e$ converges on all such inputs.
To ease the notation, let $f(e) = j$.
By the claim, there exists a stage $s$ such that for all $t \ge s$, $x^s_j = x^t_j$.
Since $\Phi^X_e (x^s_j,j) \downarrow$, if $\delta^s_e = 0$ then for some $t \ge s$ either $\mathsf{R}_e$ or some higher priority requirement must act at stage $t$.
It follows that $x^{t+1}_j$ is updated to a new value, which contradicts the claim that it already stabilized at stage $s$.
It follows that $\delta^s_e = 1$, which means that $\mathsf{R}_e$ already acted at a previous stage, say $u_e$.
This means that $\Phi^X_{e,u_e}(x^{u_e}_j,j) \downarrow$ and $e$ is the least, so it is the one to act.
If $\Phi^X_{e,u_e}(x^{u_e}_j,j) \downarrow = 0$ then we stipulated $x^{u_e}_j \lhd_j z^{u_e}_j$ and we put $z^{u_e}_j \in I_j$.
In this case, either $\Phi^X_e(z^{u_e}_j,j) \downarrow = 0$ and meet clause \ref{p2}, or $\Phi^X_e(z^{u_e}_j,j) \downarrow = 1$ and we meet clause \ref{p4}.
The proof in the case $\Phi^X_{e,u_e}(x^{u_e}_j,j) \downarrow = 1$ is analogous: if $\Phi^X_e(z^{u_e}_j,j) \downarrow = 1$ we meet clause \ref{p3}, while if $\Phi^X_e(z^{u_e}_j,j) \downarrow = 0$ we meet clause \ref{p4}.

We conclude that the $X$-computable instance $(L_j, I_j, F_j)_{j < k}$ of $\mathsf{LS}^*$ has no $\Delta_1(X)$-solution, because for each index $e$, the $X$-computable function $\Phi^X_e$ fails to be a separator set for $L_{f(e)}, I_{f(e)}, F_{f(e)}$.
\end{proof}

We can strengthen Theorem \ref{lowerbound} by refining its proof and obtain that the failure of $\I\SI{0}{2}$ suffices for its conclusion.

\begin{theorem}\label{lowerbound2}
Every model of $\RCA+\neg \I \SI02$ contains an instance $Z$ of $\mathsf{LS}^*$ with no $\Delta_1 (Z)$-definable solution. 
\end{theorem}

\begin{proof}
Let $(M,S) \vDash \RCA+\neg \I \SI02$.
We may also assume $(M,S) \vDash \B \SI02$, otherwise Theorem \ref{lowerbound} immediately yields the thesis.
All these hypothesis imply, by \cite[Subsection 2.2]{chongwong2025}, that for some $X \in S$ there are a $\Sigma^0_2(X)$-definable proper cut $I \subset M$ bounded by some $k \in M$ and a $\Sigma^0_2(X)$-definable strictly increasing cofinal function $f \colon I \rightarrow M$.
Let $(f_s)_{s \in M}$ be a $\Delta^0_1(X)$-sequence of functions such that for each $s \in M$ the map $f_s \colon k \rightarrow M$ is non decreasing, for each $x < k$ the map $s \mapsto f_s(x)$ is non decreasing, and for $x \in I$ $\lim_{s \rightarrow \infty} f_s(x) = f(x)$.

Let $(\Phi^X_e)_{e \in M}$ be an enumeration of the $X$-computable functions of the model (regarded as $\{0,1\}$-valued for convenience).
We write $\Phi^X_{e,s}$ to denote $\Phi^X_e$ after $s$ stages of the computation.

We construct by stages a $\Delta_1(X)$-instance $Z = (L_j, I_j, F_j)_{j < k}$ of $\mathsf{LS}^*$ such that for each $e \in M$, if $j \in I$ is such that $f(j-1) \leq e < f(j)$ (assuming $f(-1) = 0$) then $\{x :\Phi^X_e(x,j)=1\}$ is not a separator set for $L_j, I_j, F_j$.
Since $f$ is cofinal in $M$, the disjoint intervals $[f(j-1), f(j))$ cover all $M$.

The support of each $L_j$ is $M$ and the construction keeps parameters $\delta^s_e$ for $e,s \in M$ and $x^s_j$ for $j < k$ and $s \in M$ as in the proof of Theorem \ref{lowerbound}.
Analogously, we denote by $L_j^s, I_j^s, F_j^s$ the finite approximations of $L_j, I_j, F_j$ at the beginning of stage $s$.
We start by setting $\delta^0_e = 0$, $x^0_j = 0$, $L_j^0 = \{x^0_j\}$ and $I_j^0 = F_j^0 = \emptyset$. 
At each stage we add finitely many elements to $L_j$, define the comparabilities involving the new elements and decide which of them belong to $I_j$ or $F_j$.

For each $e \in M$ we want to satisfy the same requirement $\mathsf{R}_e$ as in the proof of Theorem \ref{lowerbound} consisting of the disjunction of the following clauses where $e \in [f(j-1), f(j))$:
\begin{enumerate}[leftmargin=*,label=(\arabic*)]
    \item \label{p11} $\Phi^X_e(x,j) \uparrow$ for some $x \in M$;
    \item \label{p22} $\Phi^X_e(x,j) \downarrow = 0$ for some $x \in I_j$;
    \item \label{p33} $\Phi^X_e(x,j) \downarrow = 1$ for some $x \in F_j$;
    \item \label{p44} $\Phi^X_e(x,j) \downarrow = 1$ and $\Phi^X_e(y,j) \downarrow = 0$ for some $x,y \in L_j$ with $y \lhd_j x$ (here $\lhd_j$ denotes the order relation on $L_j$).
\end{enumerate}
Notice that if $\mathsf{R}_e$ is satisfied, then $\Phi^X_e$ is not the characteristic function of a separator set for $L_j,I_j,F_j$.
We set $\delta^s_e = 1$ if at stage $s$ if the requirement $\mathsf{R}_e$ appears to be satisfied.
Since we do not have access to the function $f$, we use its $\Delta^0_1(X)$-approximation $(f_s)_{s \in M}$ and work with the intervals $[f_s(j-1),f_s(j))$.
As these intervals change during the construction we may have to set back $\delta^t_e = 1$ at some later stage $t>s$, unlike in the proof of Theorem \ref{lowerbound}.

\smallskip

\textit{Stage $s$}. 
For every $j < k$ we first check if $f_{s}(j-1) = f_{s+1}(j-1)$ and $f_{s}(j) = f_{s+1}(j)$.
If this is the case, we check if for some $e \in [f_{s+1}(j-1),f_{s+1}(j))$, $\Phi^X_{e,s} (x^s_j, j) \downarrow$ and $\delta^s_e = 0$.
Let $z^s_j$ and $w^s_j$ be the least elements of $M$ not in $L^s_j$.
Then we proceed exactly as in the stage by stage construction of the proof of Theorem \ref{lowerbound}.
Notice that we are also setting values for $x^{s+1}_j$ and for some of the $\delta^{s+1}_e$.

Otherwise the intervals $[f_s(j-1), f_s(j))$ and $[f_{s+1}(j-1),f_{s+1}(j))$ are different.
In this case, we let $L_j^{s+1} = L_j^s$ (and hence also $I_j^s$ and $F_j^s$ are unchanged) and $x_j^{s+1} = x_j^s$. 
Moreover, for each $e \in [f_{s+1}(j-1),f_{s+1}(j))$ we set $\delta^{s+1}_e = 0$. 

After having dealt with every $j < k$, for each $e$ such that $\delta_e^{s+1}$ is still undefined, we set $\delta_e^{s+1} = \delta^s_e$ and we move to the next stage.
This completes the construction.

\smallskip

Since the sequence $(f_s)_{s \in M}$ approximates the function $f$, we know that for each $j \in I$ there exists a stage $t_j$ such that for all $t' > t_j$, $f_{t'} (j) = f(j)$.
This means that the interval $[f_s(j-1), f_s(j))$ eventually stabilizes.
Once this happens, the linear order $L_j$ works toward requirement $\mathsf{R}_e$ for each $e \in [f(j-1), f(j))$ and, since this interval never changes anymore, the stage by stage construction coincides with that in the proof of Theorem \ref{lowerbound}.
The proof that all the requirements are satisfied is exactly the same.

We conclude that the $X$-computable instance $(L_j, I_j, F_j)_{j < k}$ of $\mathsf{LS}^*$ has no $\Delta_1(X)$-solution, because for each index $e$, the $X$-computable function $\Phi^X_e$ fails to be a separator set for some $L_j, I_j, F_j$ with $j \in I$.
\end{proof}

\begin{corollary}[$\RCA$]\label{reversemathresult}
\lq\lq $\mathsf{LS}^*$ is computably true\rq\rq\ is equivalent to $\I \SI02$.
In particular $\RCA + \B\SI02 \nvdash \mathsf{LS}^*$ and consequently also $\RCA + \B\SI02 \nvdash \DBpp$.
\end{corollary}

\begin{proof}
Lemma \ref{leastsepel} yields immediately that $\I \SI02$ proves that $\mathsf{LS}^*$ is computably true.
Conversely, by the relativized version of Theorem \ref{lowerbound2}, every model of $\neg \I \SI02$ has an instance $Z$ of $\mathsf{LS}^*$ with no $\Delta_1 (Z)$-solutions.
This witnesses that \lq\lq $\mathsf{LS}^*$ is computably true\rq\rq\ does not hold in the model.

For the second part, let $M$ be a first order model of $\B\Sigma_2 + \neg \I \Sigma_2$.
If we take $S$ the set of $\Delta_1$-definable subsets of $M$, then $(M,S) \vDash \RCA + \B\SI02 + \neg \I \SI02$ and by Theorem \ref{lowerbound2} in $(M,S)$ there is a computable instance $Z$ of $\mathsf{LS}^*$ with no $\Delta_1$-solutions.
We conclude that in $(M,S)$ there is no solution for $Z$.
\end{proof}

Corollary \ref{reversemathresult} implies that $\mathsf{LS}$ has no uniform proof in $\RCA + \B\SI02$, as otherwise $\mathsf{LS}^*$ would be provable in $\RCA + \B\SI02$.

\appendix
\section{A result in Weihrauch reducibility}

The origins of Weihrauch reducibility go back to the late 1980s.
Later, it was noticed in \cite{GHERARDI200885}, and independently in \cite{dorais2016uniform}, that this notion can also be a fruitful way of comparing problems that arise in the context of reverse mathematics.
A recent survey is \cite{BGP17}.

The statement $\mathsf{LS}$ can be easily viewed as an instance-solution problem.
An instance consists of a linear order and two subsets of its domain such that all elements of the first set are below all elements of the second set.
A solution of this instance is a separator set for this triple.
We still use $\mathsf{LS}$ to denote this problem.

The problem $\mathsf{XC}_1$ is the restriction of the choice function $\mathsf{C}_{[0,1]}$ to convex subsets of $[0,1]$ (i.e.\ intervals): an instance is a nonempty convex closed subset $A$ of $[0,1]$, represented by open intervals with rational endpoints which cover its complement, and the solution is any point of $A$. 
The Weihrauch degree of $\mathsf{XC}_1$ was studied in \cite{3,2}.

\begin{proposition}\label{LSWei}
$\mathsf{XC}_1 \equiv_{\mathrm{W}} \mathsf{LS}$.
\end{proposition}

\begin{proof}
First we show that $\mathsf{XC}_1 \leq_{\mathrm{W}} \mathsf{LS}$. A special case of \cite[Proposition~3.4]{1} shows that from a convex closed set $A \subseteq [0,1]$ we can compute sequences of rational numbers $a_0 < a_1 < \dots$ and $b_0 > b_1 > \dots$ such that
\[
\sup\{a_n \mid n \in \mathbb{N}\} = \inf A
\quad\text{and}\quad
\inf\{b_n \mid n \in \mathbb{N}\} = \sup A.
\]
We use $\mathbb{Q} \cap [0,1]$ as the linear order, $\{a_n : n \in \mathbb{N}\}$ as the lower and
$\{b_n : n \in \mathbb{N}\}$ as the upper set of an input to $\mathsf{LS}$. 
The output of $\mathsf{LS}$ on this input is a Dedekind cut of some real number $x \in A$, which concludes the reduction.

To see that $\mathsf{LS} \leq_{\mathrm{W}} \mathsf{XC}_1$, we start with a linear order $L$, the lower set $I$ and the upper set
$F$. We can compute an order embedding
\[
\iota \colon L \times 2 \to \mathbb{Q} \cap [0,1],
\]
where $L \times 2$ is ordered lexicographically.
Let
\[
A := \{x \in [0,1] : (\forall a \in I\ \iota(a,1) \le x) \wedge (\forall b \in F\ \iota(b,0) \ge x)\}.
\]
We can compute $A$ as a closed set, and it is by construction an interval. We use $\mathsf{XC}_1$ to obtain some $x \in A$.

Now we construct a separating set $S \subseteq L$ by identifying for each $\ell \in L$ a true case among
$\iota(\ell,0) < x$ and $\iota(\ell,1) > x$. As $\iota(\ell,0) < \iota(\ell,1)$, at least one of the cases must hold.
In the former case, we let $\ell \in S$, in the latter we set $\ell \notin S$.
\end{proof}

The reason in the above proof we work with $L \times 2$ instead of $L$ is because it may happen that $x \in A$ is in the range of $\iota$. 
Mapping every element of $L$ to two distinct rationals ensures that we can compare at least one of them to any real number.

Proposition \ref{LSWei} implies that $\mathsf{LS}$, as problem, is not computable.
This leads again to the conclusion that $\mathsf{LS}$, as a reverse mathematics statement, has no uniform proof in $\RCA$.

\end{document}